\documentclass[smallextended]{svjour3}       
\smartqed  

\usepackage{adjustbox}
\usepackage{graphicx}
\usepackage{amsmath}
\usepackage{amssymb}
\usepackage{multirow}
\usepackage{url}
\usepackage[tight,footnotesize]{subfigure}

\usepackage[switch, modulo]{lineno}
\usepackage{setspace}
\usepackage{amsmath}
\DeclareMathOperator*{\argmax}{argmax} 

\begin{document}
\title{Recursive Modified Pattern Search on High-dimensional Simplex : A Blackbox Optimization Technique}


\author{Priyam Das}


\institute{Priyam Das \at
              University of Texas MD Anderson Cancer Center, USA \\
               Tel.: +1 919-308-5892\\
              \email{pdas@ncsu.edu}}

\date{Received: date / Accepted: date}

\maketitle

\begin{abstract}
In this paper, a novel derivative-free pattern search based algorithm for Black-box optimization is proposed over a simplex constrained parameter space. At each iteration, starting from the current solution, new possible set of solutions are found by adding a set of derived step-size vectors to the initial starting point. While deriving these step-size vectors, precautions and adjustments are considered so that the set of new possible solution points still remain within the simplex constrained space. Thus, no extra time is spent in evaluating the (possibly expensive) objective function at infeasible points (points outside the unit-simplex space); which being the primary motivation of designing a customized optimization algorithm specifically when the parameters belong to a unit-simplex. While minimizing any objective function of $m$ parameters, within each iteration, the objective function is evaluated at $2m$ new possible solution points. So, upto $2m$ parallel threads can be incorporated which makes the computation even faster while optimizing expensive objective functions over high-dimensional parameter space. Once a local minimum is discovered, in order to find a better solution, a novel `re-start' strategy is considered to increase the likelihood of finding a better solution. Unlike existing pattern search based methods, a sparsity control parameter is introduced which can be used to induce sparsity in the solution in case the solution is expected to be sparse in prior. A comparative study of the performances of the proposed algorithm and other existing algorithms are shown for a few low, moderate and high-dimensional optimization problems. Upto 338 folds improvement in computation time is achieved using the proposed algorithm over Genetic algorithm along with better solution. The proposed algorithm is used to estimate the simultaneous quantiles of North Atlantic Hurricane velocities during 1981--2006 by maximizing a non-closed form likelihood function with (possibly) multiple maximums.
\end{abstract}
\newpage
\section{Introduction}
\label{sec1}
Black-box can be described as a device, system or an object which can be observed only in terms of inputs and outputs. However, the ongoing process within it, is considered unknown. Black-box objective function can be considered similar to any
\begin{figure*}[ht]
\includegraphics[width=0.9\textwidth]{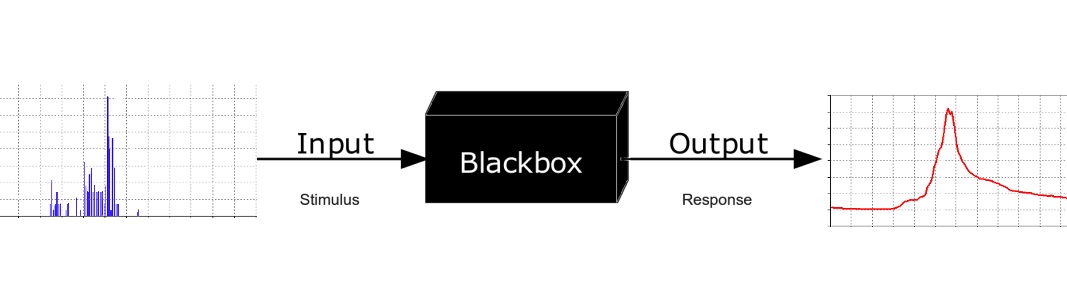}
\caption{Blackbox mechanics.}
\label{figure_00} 
 \end{figure*}
 other Blackbox device where for any given input of the values of the parameters, only the value of the objective function is observed without any further knowledge about the structure, continuity or differentiability of that objective function. 
Now, consider the minimization problem 
\begin{align}
minimize \;:& \; f(\mathbf{p}), \text{ where }\mathbf{p} = (p_1,\cdots,p_m) \nonumber \\
subject \; to \; :& \; p_i \geq 0, \; 1 \leq i \leq m, \; \sum_{i=1}^mp_i = 1, \mathbf{p} \in \mathrm{R}^m,
\label{opt_problem}
\end{align}
where $f(\mathbf{p})$ might have points of discontinuity, multiple local minimums and might not be differentiable over the domain. In the field of computational mathematics, statistics and operational research, optimization problems on the simplex parameter space are quite common. Some of the useful and convenient methods like modeling with B-splines (e.g., monotonic function estimation technique proposed in \cite{Das2017a},\cite{Das2017b},\cite{Das2018}), estimation of parameters in multinomial problems, estimation of Markov chain transition matrix, estimation of mixture proportions of mixture distribution (e.g., \cite{Basford1985}) are a few examples where the parameter space is given by a unit-simplex or a collection of unit-simplexes. 

In literature, a variety of methods can be found for optimizing linear and non-linear objective functions on constrained linear space, therefore they can be used for optimizing objective functions on unit-simplex constrained parameter space as well. In practice, convex optimization algorithms (e.g., `Interior-point (IP)' algorithm, see \cite{Potra2000}, \cite{Karmakar1984}, \cite{Boyd2006}, \cite{Wright2005}; `Sequential Quadratic Programming (SQP)', algorithm see \cite{Wright2005}, \cite{Nocedal2006}, \cite{Boggs1996}) are widely used to minimize non-linear objective functions on constrained and unconstrained parameter spaces.  However, in case the objective function is non-convex with multiple minimas, the convex optimization algorithms might get stuck at a local minimum and return it as the solution without committing any further attempt to find a better minimum than the obtained one. To improve the likelihood of obtaining better solution using convex optimization methods, one possible strategy is to start any convex optimization from several starting points and to choose the best solution out of them. For low dimensional non-convex optimization problems, this strategy of starting from multiple initial points might be computationally affordable. However, as the dimension of the parameter space increases, this strategy proves to be computationally very expensive. 

In order to globally minimize any objective function with (possibly) multiple minimums, in the last few decades, many deterministic and non-deterministic (i.e., stochastic global search algorithms) global optimization strategies have been proposed which can also be extended or applied to minimize functions of linearly constrained parameter spaces (\cite{Rios2013}). Among non-deterministic global minimization algorithms, `Genetic algorithm (GA)' (see \cite{Fraser1957}, \cite{Bethke1980}, \cite{Goldberg1989}) and `Simulated annealing (SA)' (see \cite{Kirkpatrick1983}, \cite{Granville1994}) are widely used in different fields. However there remain a few drawbacks of these methods. e.g., GA does not scale well with complexity as in higher dimensional optimization problems there is often an exponential increase in search space size (see \cite{Geris2012}, page 21). Besides, another major problem with these two methods is that they might prove to be much expensive even if the objective function is convex. Now, it can be argued that it is not conventional to use global optimization technique to minimize convex functions in case we already have that prior knowledge about its convexity. However, in case the function is actually convex but the true structure of the function (i.e., a convex Blackbox function) is not known, ideally Blackbox and global optimization techniques should be applied to minimize it. In that scenario, the Blackbox techniques (e.g., GA) might prove to be computationally expensive even though the objective function is actually convex. Among other non-deterministic global optimization techniques,  Particle Swarm Optimization (PSO) (\cite{Kennedy1995}, \cite{Eberhart1995}) remains popular for unconstrained global optimization.

With the increasing access to high-performance modern computers and clusters (\cite{Hilbert2011}), some of the existing parallelizable optimization algorithms (e.g., Monte Carlo methods) have a great advantage for certain types of problems. The motivation behind using parallelization in these methods is mainly to either start from different starting points or to use different random number generator seeds simultaneously (\cite{Kerr2014}). As mentioned earlier, though these methods perform well for lower dimensional parameter spaces, since with an increasing number of dimensions, parameter space grows exponentially, the way these methods use parallelization, might not be very helpful. 
If an algorithm is designed in such a way that the requirement of parallelization increases linearly with the dimension of the parameter space, it would be more convenient and useful for handling constraint optimization problem over high-dimensional parameter spaces.

\begin{figure}[]
	\centering
	\includegraphics[width=0.8\textwidth]{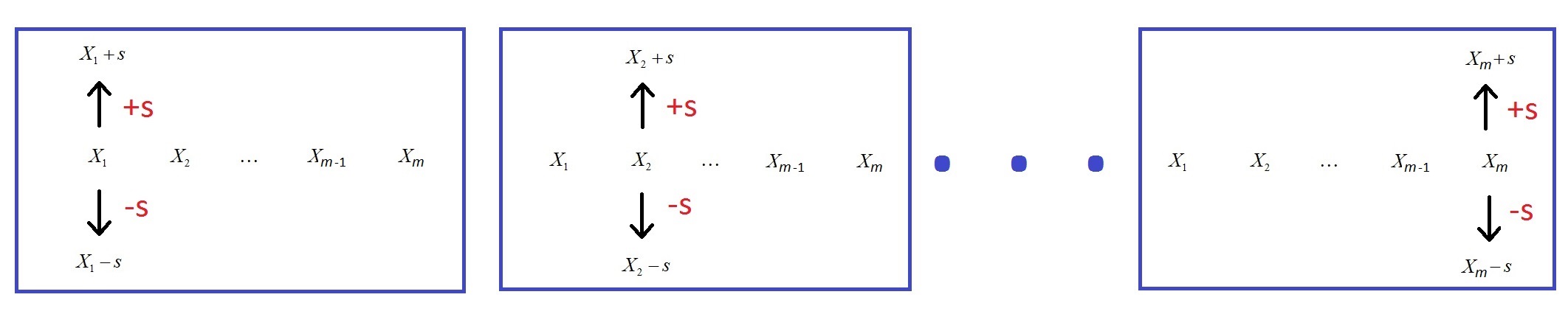}
	\caption{Fermi's principle : Possible $2m$ movements starting from any initial point inside an iteration with fixed step-size $s$ while minimizing any $m$-dimensional unconstrained objective function.}
	\label{fermi}
\end{figure}

In order to minimize any unconstrained Black-box function, \cite{Fermi1952} proposed an effective coordinate search based algorithm (known as Fermi's principle, see Figure \ref{fermi}), where at each iteration, a step-size $s>0$ (a scalar) is fixed based on some criteria. Then, in order to look for candidate solutions (i.e., set of points obtained around the current solution using any algorithm, where the objective function values are to be computed while looking for a better solution than the current one), this step-size $s$ is added and subtracted from each coordinate of the current solution one at a time while keeping the other coordinates unchanged. Thus, $2m$ new candidate solutions are generated at each iteration. Out of these $2m+1$ points (including the starting point and $2m$ new candidate solutions), the point with minimum objective function value is considered as the current updated solution from where the next iteration begins with a step-size (maybe same or different from step-size considered in previous iteration). Extending this idea in general, \cite{Hooke1961} introduced Direct search algorithms for unconstrained optimization problems. Direct search is a technique to solve optimization problem without using any information regarding the gradient of the objective function. Rather, at each iteration it evaluates the objective function a set of points around the current solution looking for a better solution. Generalizing the idea of \cite{Hooke1961}, \cite{Torczon1997} proposed `Generalized Pattern Search' (GPS) on unconstrained parameter space. In GPS, using an \textit{exploratory moves algorithm} (\cite{Torczon1997}), a set of step-size vectors are generated such that adding each of the set of step-size vectors to the current solution yields one candidate solution. Thus, based on the values of the set of step-size vectors, new set of candidate solutions are obtained by making coordinate-wise movements around the current solution point in order to find a better solution. Further generalizing and modifying GPS, \cite{Kolda2003} and \cite{Audet2006} proposed `Generating Set Search'(GSS), `Mesh Adaptive Direct Search'(MADS) methods respectively. A few related articles on Direct search and GPS methods were proposed in \cite{Custodio2015}, \cite{Martinez2013}, \cite{Lewis1999}. Some other derivative-free optimization methods can be found in  \cite{Audet2014}, \cite{Conn2009}, \cite{Jones1998}, \cite{Digabel2011}, \cite{Martelli2014}, \cite{Audet2008} and \cite{Audet22008}.  However, most of these methods were developed for unconstrained optimization problems. Though some of them were designed for linearly constrained spaces (e.g., \cite{Lewis2000} ) or hyper-rectangular space (e.g., \cite{Das2016}), however, to the best of our knowledge, no article has proposed derivative-free Black-box optimization technique specifically designed for simplex-constraint space.

\subsection{Heuristic idea of Recursive Modified Pattern Search on Simplex (RMPSS)}
In order to design an algorithm to minimize a Black-box function over unit simplex constrained parameter space, we adopt the basic idea of Fermi's principle (\cite{Fermi1952}). However, unlike their case, since we are minimizing over a simplex constrained space, updating only one coordinate (by adding or subtracting step-size $s>0$) at a time would result in generating candidate solutions outside the unit simplex. Therefore, once step-size $s$ is added (or subtracted) to any coordinate, to keep the sum of the coordinates constant (i.e., 1), $\frac{s}{m-1}$ is deducted (or added) from each of the rest $(m-1)$ coordinates, where $m$ denotes the dimension of the parameter vector lying on the simplex. Thus the sum of the coordinates of the new candidate solution would be 1. However, during thees update steps, a scenario might arise where atleast one updated coordinate is either $<0$ or $>1$. In order to handle those cases, some case-specific modifications are also considered (see Section \ref{sec_algo}) during the search for candidate solutions. Thus, at each iteration, $2m$ new candidate solutions are obtained within simplex constrained space making coordinate-wise movements and, similar to Fermi's principle, at the end of each iteration, out of the $2m+1$ solutions, the solution with the least objective function value is retained as the updated solution from where the next iteration begins. Note that finding $2m$ candidate solutions and computing the value of the objective function at those coordinates can be performed simultaneously using $2m$ parallel threads since these sequence of operations do not interact with each other once the step-size of the movement is fixed at the beginning of the iteration. Therefore, upto $2m$ parallel threads could be used for computation making it way faster to optimize expensive objective functions. Thus, while minimizing expensive objective function over high-dimensional simplex space ($m = 100,200, 500$ etc), the proposed algorithm allows to use the full potential of GPU computing which, nowadays, allows to run upto hundreds of parallel threads simultaneously.

In order to increase the likelihood of capturing the true global minimum, in the proposed algorithm we consider a specific deterministic way of updating the step-size $s$ throughout the iterations. The proposed algorithm can be divided into several sequences of iterations called \textit{run}s. Within each \textit{run}, the first iteration is initialized from a starting point and a large  step-size (fixed by the user). Note that larger step-size yields candidate solutions which are far away from the starting point. Thus keeping larger step-size at the beginning of a \textit{run} allows to select candidate solutions far away from the starting point. While making movements with a particular step-size, step-size is kept unchanged as long as better solution is obtained in each iteration. Once the objective function value stops improving, in order to incorporate a coarser search in the neighbourhood of the current solution, step-size $s$ is reduced to $\frac{s}{\rho}$ (where $\rho>1$ is a constant provided by the user) before further iterations are performed. Iterations are performed until the step-size becomes smaller than a step-size threshold supplied by the user. Under certain set of regularity conditions, it is shown that the solution obtained by this sequence of operations would yield a local minimum (or global minimum, depending on the regularity conditions)  if step-size threshold is taken to be arbitrarily small (\cite{Lewis2000}, also see Section \ref{sec_theory}). While minimizing a non-convex function with this strategy, it is possible that the solution returned at the end of a \textit{run} is a local minimum despite existence of a better solution. In the field of heuristic global optimization algorithms, once a local minimum is identified, it is a common strategy to consider generating candidate solutions far away from the current solution in order to jump out of the (possibly locally convex) local neighbourhood of the current solution (e.g., GA, SA). To incorporate this strategy in the proposed algorithm, another \textit{run} is initiated with large step-size starting from the solution returned by the previous \textit{run}. Larger value of step-size ensures that the new candidate solutions are far away from the current solution which increases the likelihood of finding a better solution in case the function is locally convex near the solution point despite existence of better solution in different neighbourhood. \textit{Run}s are performed until two consecutive \textit{run}s yield the same solution (see Figure \ref{flowchart}).

\begin{figure}[]
	\centering
	\includegraphics[width=0.8\textwidth]{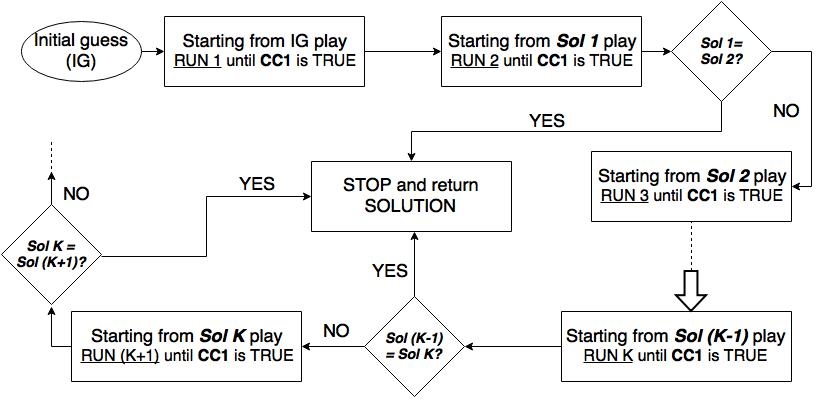}
	\caption{Flowchart of the RMPSS algorithm. CC1 : Convergence criteria 1 given by step (8) of STAGE 1 in Section \ref{sec_algo}.}
	\label{flowchart}
\end{figure}

Another novel feature of the proposed algorithm is incorporation of sparsity parameter in order to encourage sparse solution. The sparsity parameter is supplied by the user and depending on the prior guess about the sparsity of the solution, the value of the sparsity parameter can be adjusted accordingly which is useful for several statistical and computational problems (e.g., estimating proportions of mixture model for small sample and high number of existing population classes). Based on the way candidate solutions are chosen within each iteration, the proposed algorithm can be easily recognized as a variation of Pattern search based methods (e.g., GPS in \cite{Torczon1997}) which is specifically designed to minimize objective functions on a simplex constrained space. However, unlike existing Pattern search methods, in the proposed algorithm \textit{run}s are performed repeatedly until the final solution is obtained. Also, some case-specific modification of step-sizes are considered during update step to keep the new set of candidate solutions on simplex. So, the proposed method is named `Recursive Modified Pattern Search on Simplex' (RMPSS).

The rest of the article is arranged as follows. Section 2 describes the RMPSS algorithm and the roles of the considered tuning parameters. In Section 3, we show the convergence property of RMPSS under a set of regularity conditions on the objective function. In Section 4 it is shown how a few other well-known constrained optimization problems can be solved with RMPSS algorithm. In Section 5, we perform extensive comparative study of the performances of RMPSS and a few other well-known algorithms for a wide range of objective functions. In Section 6, RMPSS is applied to a real data, to estimate simultaneous quantiles of North Atlantic Hurricanes by maximizing a (possibly) multimodal likelihood whose parameters belong to simplexes. A case-study of RMPSS using parallel computing is also considered in this section. In Section 7, a discussion on RMPSS is provided.

\section{Algorithm}
\label{sec_algo}
Consider the optimization problem defined by Equation \eqref{opt_problem}. 
Define $$\mathbf{S} = \{\mathbf{p} = (p_1,\ldots,p_m) \in \mathbb{R}^m \; | \; \sum_{i=1}^mp_i = 1, \, p_i \geq 0, \, 0 \leq i \leq m \}$$ Now the problem can be written as 
\begin{align}
minimize \;:& \; f(\mathbf{p}) \nonumber \\
subject \; to \; : & \; \mathbf{p} \in \mathbf{S}.
\label{opt_problem_3}
\end{align}


\subsection{Parameters in RMPSS algorithm }
As mentioned in Section \ref{sec1}, the proposed algorithm consists of several \textit{runs}. Each \textit{run} is an iterative procedure and requires a starting point. In order to find a better solution than the starting point, a sequence of iterations are performed inside each \textit{run}. A \textit{run} stops based on some convergence criteria which is discussed elaborately in the following. At the end of each \textit{run}, a solution (which may or may not be the final solution depending on other criteria) is returned. The initial guess of the solution should be entered by the user and it is considered  as the starting point for the first \textit{run}. Second \textit{run} onwards, each \textit{run} automatically takes the solution of the previous \textit{run} as the starting point of that \textit{run}. Operations performed inside each \textit{run} attempt to minimize the objective function. Hence, the solution tends to improve after each \textit{run}. Once two consecutive \textit{runs} yield the same solution, the algorithm terminates and returns the solution obtained at the last \textit{run} as the final solution to the user.

Each \textit{run} is a similar iterative procedure except the fact that the values of tuning parameters after each \textit{run} might be changed. In each \textit{run} there are four tuning parameters which are \textit{initial global step size} $s_{initial}$, \textit{step decay rate} $\rho$, \textit{step size threshold} $\phi$ and \textit{sparsity threshold} $\lambda$. Except \textit{step decay rate} $\rho$, the values of the other tuning parameters are set before starting the algorithm and kept unchanged till the algorithm converges. The value of \textit{step decay rate} is taken to be $\rho=\rho_1$ for the first \textit{run} and $\rho=\rho_2$ for the following \textit{runs}. Overall there are 5 tuning parameters which are $s_{initial}, \rho_1, \rho_2, \phi$ and $\lambda$. Apart from these parameters, we consider two more quantities namely $max\_iter$, which denotes the maximum number of allowed iterations inside a \textit{run}, and $max\_runs$ which denotes the maximum number of \textit{runs} to be executed.

Inside each \textit{run}, there are a few parameters whose values are updated throughout the iterations which are \textit{global step size} (denoted by $s^{(j)}$ for $j$-th iteration within a \textit{run}) and \textit{local step sizes} (within any iteration of a \textit{run}, there are $2m$ of them, denoted by $\{s_i^+\}_{i=1}^m$ and $\{s_i^+\}_{i=1}^m$). 
In the first iteration of each \textit{run}, we set initial value of \textit{global step size} $s^{(1)} = s_{initial}$. Let $s^{(j)}$ denote the value of the \textit{global step size} at $j$-th iteration. . The value of \textit{global step size} is kept unchanged throughout an iteration. But at the end of each iteration, based on some convergence criteria (see below, see step (7) of STAGE 1) its value is either kept unchanged or divided by a factor \textit{step decay rate} $\rho$. Hence, in the $(j+1)$-th iteration, the \textit{global step size} is $s^{(j+1)}$ would be equal to either $s^{(j)}$ or $\frac{s^{(j)}}{\rho}$. At the beginning of any iteration, the \textit{local step sizes} are set equal to the current \textit{global step size}. For example, at the beginning of $j$-th iteration, we set $s_i^+ = s_i^- = s^{(j)}$ for $i=1,\ldots,m$. Suppose the current value of $\mathbf{p}$ in the $j$-th iteration is $\mathbf{p}^{(j)} = (p_1^{(j)}, \ldots, p_m^{(j)}) \in \mathbf{S}$. During the $j$-th iteration, the objective function value is evaluated at  $2m$ points within the domain $\mathbf{S}$ which are obtained by making coordinate-wise movements depending on the \textit{local step sizes} starting from the current solution $\mathbf{p}^{(j)}$. The value of the \textit{local step sizes} are subject to be updated if the movements corresponding to those step sizes yield points outside $\mathbf{S}$ (see step (3) and (4) of STAGE 1). Note that the \textit{global step size} is dependent on the iteration number whereas the \textit{local step sizes} are not dependent on the iteration number as their values are not related to their values in the previous iteration. 

While solving problems on high-dimensional simplex constrained parameter space, in order to encourage sparsity, we incorporate another sparsity parameter named \textit{sparsity threshold}, denoted by $\lambda$. In the process of updating the solution at each iteration, once the value of a co-ordinate (of the solution) goes below the \textit{sparsity threshold} $\lambda$, we consider those co-ordinates to be `insignificant'. 
Suppose $l$-th component of $\mathbf{p}^{(j)}$ is `insignificant'. Then, while making coordinate-wise movements to find new candidate solutions around $\mathbf{p}^{(j)}$, $l$-th component of $\mathbf{p}^{(j)}$ i.e., $p_l^{(j)}$ is kept unchanged
(see step (3) and (4) of STAGE 1). 

After starting from current solution $\mathbf{p}^{(j)}$ at the $j$-th iteration of a \textit{run}, suppose the objective function values at the candidate solution points are given by $\{f_i^+\}_{i=1}^m$ and $\{f_i^-\}_{i=1}^m$ (see step (3) and (4) of STAGE 1). We find the smallest one out of these $2m$ values. If the smallest of these $2m$ values is smaller than $f(\mathbf{p}^{(j)})$, the point corresponding to that smallest value of the objective function is accepted and updated as the current solution. Then the `insignificant' positions (if any) are replaced by 0 and the sum of the `insignificant' positions (named `$garbage$') is divided by the number of `significant' positions and that quantity is added to those remaining positions so that sparsity is encouraged and the simplex constraint is maintained. Now this new point is considered as the new solution and the value of $\mathbf{p}^{(j+1)}$ is set equal to it (see step (5) and (6) of STAGE 1). In order to decide whether the solutions obtained by two consecutive iterations are close enough, the square of the euclidean distance of the objective function parameters of two consecutive iterations are computed. If that comes to be less than $tol\_fun$, the \textit{global step size} is divided by a factor $\rho$, the \textit{step decay rate} (see step (7) of STAGE 1). A \textit{run} ends when the \textit{global step size} becomes less than or equal to \textit{step size threshold} $\phi$ (see step (8) of STAGE 1). Once same solution is returned by two consecutive \textit{runs}, our algorithm terminates and returns the solution obtained in the last \textit{run} as the final solution.

\begin{table}[]
  \resizebox{1\columnwidth}{!}{
\begin{tabular}{llll}
\hline
Parameter & Description & Role & \begin{tabular}[c]{@{}l@{}}Recommended values\\ and comments\end{tabular} \\ \hline
$s_{initial}$ & \begin{tabular}[c]{@{}l@{}}initial global step\\ size\end{tabular} & \begin{tabular}[c]{@{}l@{}}Initial step-size at the beginning of\\ the run, higher value promotes\\ selection of distant candidate solutions.\end{tabular} & \begin{tabular}[c]{@{}l@{}}1 (setting it 1 allows maximum possible\\ jump within simplex space)\end{tabular} \\ \hline
$\rho$ & step decay rate & \begin{tabular}[c]{@{}l@{}}Controls the rate of decay of global step-\\ size, smaller value of $\rho$ results in\\ slower decay of the global step size, \\ thus it allows denser search in the \\ neighborhood of the current solution at\\ the expense of higher computation time.\end{tabular} & \begin{tabular}[c]{@{}l@{}}has to be $>1$ \\ $\rho_1 = 2$ (for first run)\\ $\rho_2 = 1.05$ (second run onwards)\end{tabular} \\ \hline
$\phi$ & \begin{tabular}[c]{@{}l@{}}lower bound of \\ global step size\end{tabular} & \begin{tabular}[c]{@{}l@{}}Controls precision of search, smaller \\ value of $\phi$ results in more accurate \\ solution in the expense of higher \\ computation time.\end{tabular} & \begin{tabular}[c]{@{}l@{}}$10^{-3}$ (can be taken upto as small\\ as $10^{-7}$ for more accurate solution)\end{tabular} \\ \hline
$\lambda$ & sparsity threshold & \begin{tabular}[c]{@{}l@{}}(i) Controls sparsity, encourage sparse \\ solution.\\ (ii) Helps in the search procedure when\\ coordinate(s) of the starting point of any\\ iteration is(are) close to 0.\end{tabular} & \begin{tabular}[c]{@{}l@{}}$10^{-3}$ (may consider $10^{-2}$ for inducing \\ more sparsity, or can be set as \\ small as $10^{-7}$)\end{tabular} \\ \hline
tol\_fun & \begin{tabular}[c]{@{}l@{}}termination tolerance \\ on the function value\end{tabular} & \begin{tabular}[c]{@{}l@{}}The minimum amount of improvement \\ in objective function value required so\\  that the global step size is not decayed.\end{tabular} & $10^{-15}$ \\ \hline
max\_runs & max no. of runs & Put an upper limit on number of runs. & \begin{tabular}[c]{@{}l@{}}1000 (however the algorithm \\ converged before 1000 runs in all\\  the cases considered in this article)\end{tabular} \\ \hline
max\_iter & max no. of iterations & \begin{tabular}[c]{@{}l@{}}Put an upper limit on number of \\ iterations allowed within each run.\end{tabular} & \begin{tabular}[c]{@{}l@{}}50000 (however required number of\\  iterations within a run never crossed\\ 50000 in any of the considered cases)\end{tabular} \\ \hline
\end{tabular}}
\caption{Tuning parameters and their roles in the RMPSS algorithm.}
\label{tuning}
\end{table}

The default value of $s_{initial}$ is taken to be equal to $1$ which is the maximum possible jump size of one coordinate of a parameter belonging to simplex. $\rho_1$ and $\rho_2$ denote the \textit{step decay rates} for the first \textit{run} and the following \textit{runs} respectively. Taking smaller value of \textit{step decay rate} results in slower decay of the \textit{global step size}, thus it allows denser search in the neighbourhood of the current solution at the expense of higher computation time. Based on experiments over a wide range of challenging low, moderate and high dimensional objective functions on simplex constrained spaces, we note that setting the default values of these parameters $\rho_1 = 2$ and $\rho_2 = 1.05$ yields reasonable solution outperforming a few competing well-known algorithms (as shown in simulation study). $\phi$ denotes the lower bound of the \textit{global step size} for movement within the simplex constrained domain. Making the value of $\phi$ smaller results in more accurate solution at the cost of higher computation time (shown in simulation study). It's default value is taken to be equal to $10^{-3}$. Sparsity adjusting parameter $\lambda$ controls the sparsity. At the end of each iteration, the positions of the current estimated parameter with values less than $\lambda$ are set equal to $0$ (see step (6) of STAGE 1) along with an adjustment step to other coordinates so that the candidate solution still remains in simplex space. In case, the solution is expected to be sparse, it's value should be set larger and in case, the solution is not expected to be sparse, it's value should be set relatively smaller. We note that, in general, the default value of $\lambda=10^{-3}$ works fine over a range of low, medium and high-dimensional optimization problems. In order to stop infinite looping (in case if any), we set maximum number of allowed iteration within a \textit{run} $max\_iter = 50000$ and maximum number of allowed \textit{runs}s $max\_runs = 1000$. A brief summary of the parameters of RMPSS algorithm and their roles is noted down in Table \ref{tuning}.

\subsection{RMPSS algorithm}
The whole algorithm is described diving into two parts namely STAGE 1 and STAGE 2. The steps mentioned in STAGE 1 are performed within each \textit{run}; while the operations in STAGE 2 are performed in between two consecutive \textit{run}s to decide whether further execution of the follwoing \textit{run} is necessary. Before going through the STAGE 1 for the first time, we set $R=0,\rho = \rho_1$ and initial guess of the solution $\mathbf{p}^{(1)} = \mathbf{z}^{(R)}= (p_1^{(1)},\ldots, p_m^{(1)}) \in \mathbf{S}$. 
\\
\textbf{STAGE : 1}
\begin{enumerate}
\item Set $j=1$. Set $s^{(j)} = s_{initial}$ Go to step (2).
\item If $j > max\_iter$, set $\hat{\boldsymbol{p}} = \boldsymbol{p}^{(j-1)}$. Go to step (9). Else, set $s_i^+ = s_i^-=s^{(j)}$ and $f_i^+ = f_i^-=Y^{(j)}=f(\boldsymbol{p}^{(j)})$ for all $i=1,\cdots, m$. Set $i=1$ and go to step (3).
\item If $i > m$, set $i=1$ and go to step (4). Else, find $K_i^+=n(S_i^+)$ where $S_i^+ = \{l\;|\; p_l^{(j)} > \lambda,l\neq i\}$. If $K_i^+ \geq 1$, go to step (3a), else set $i=i+1$ and go to step (3).
\begin{enumerate}
\item If $s_i^+ \leq \phi$, set $i =i+1$ and go to step (3). Else (if $s_i^+ > \phi$), evaluate vector $\boldsymbol{q}_i^+ = (q_{i1}^+, \cdots, q_{im}^+)$ such that 
\begin{align*}
q_{il}^+ &= p_i^{(l)}+s_i^+ \quad \text{for} \quad l=i \\
&=p_i^{(l)}-\frac{s_i^+}{K_i^+} \quad \text{if} \quad l \in S_i^+ \\
&=p_i^{(l)} \quad \text{if} \quad l \in (S_i^+\cup\{i\})^{C}
\end{align*}
Go to step (3b).
\item Check whether $\boldsymbol{q}_i^+ \in \mathbf{S}$ or not. If $\boldsymbol{q}_i^+ \in \mathbf{S}$, go to step (3c). Else, set $s_i^+ = \frac{s_i^+}{\rho}$ and go to step (3a)
\item Evaluate $f_i^+ = f(\boldsymbol{q}_i^+)$. Set $i = i+1$ and go to step (3).
\end{enumerate}
\item If $i > m$, go to step (5). Else, find $K_i^-=n(S_i^-)$ where $S_i^- = \{l\;|\; p_l^{(j)} > \lambda,l\neq i\}$. If $K_i^- \geq 1$, go to step (4a), else set $i=i+1$ go to step (4).
\begin{enumerate}
\item If $s_i^- \leq \phi$, set $i =i+1$ and go to step (4). Else (if $s_i^- > \phi$), evaluate vector $\boldsymbol{q}_i^- = (q_{i1}^-, \cdots, q_{im}^-)$ such that 
\begin{align*}
q_{il}^- &= p_i^{(l)}-s_i^- \quad \text{for} \quad l=i \\
&=p_i^{(l)}+\frac{s_i^-}{K_i^-} \quad \text{if} \quad l \in S_i^- \\
&=p_i^{(l)} \quad \text{if} \quad l \in (S_i^-\cup\{i\})^{C}
\end{align*}
Go to step (4b)
\item Check whether $\boldsymbol{q}_i^- \in \mathbf{S}$ or not. If $\boldsymbol{q}_i^- \in \mathbf{S}$, go to step (4c). Else, set $s_i^- = \frac{s_i^-}{\rho}$ and go to step (4a)
\item Evaluate $f_i^- = f(\boldsymbol{q}_i^-)$. Set $i = i+1$ and go to step (4).
\end{enumerate}
\item Set $k_1 = \underset{1 \leq l \leq m}{\arg\min} \;f_l^+$ and $k_2 = \underset{1 \leq l \leq m}{\arg\min} \;f_l^-$. If  $ \min (f_{k_1}^+, f_{k_2}^-) < Y^{(j)}$, go to step (5a). Else, set $\boldsymbol{p}^{(j+1)} = \boldsymbol{p}^{(j)}$ and $Y^{(j+1)} = Y^{(j)}$, set $j=j+1$. Go to step (7).
\begin{enumerate}
\item If $f_{k_1}^+ < f_{k_2}^-$, set $\boldsymbol{p}_{temp} = \boldsymbol{q_{k_1}^+}$, else (if $f_{k_1}^+ \geq f_{k_2}^-$), set $\boldsymbol{p}_{temp} = \boldsymbol{q_{k_2}^-}$. Go to step (6).
\end{enumerate}
\item Find $K_{updated}=n(S_{updated})$ where $S_{updated} = \{l\;|\; \boldsymbol{p}_{temp}(l) > \lambda,l=1,\cdots,m\}$. Go to step (6a).
\begin{enumerate}
\item If $K_{updated}=m$, set $\boldsymbol{p}^{(j+1)} = \boldsymbol{p}_{temp}$, set $j=j+1$. Go to step (7). Else, go to step (6b)
\item Set $garbage = \sum_{j\in S_{updated}^C} \boldsymbol{p}_{temp}(k)$. 
\begin{align*}
\boldsymbol{p}^{(j+1)}(l) &= \boldsymbol{p}_{temp}(l) + garbage/K_{updated} &\quad \text{if} \; l\in S_{updated} \\
&=0 &\quad \text{if} \; l\in S_{updated}^C
\end{align*}
Set $j=j+1$. Go to step (7).
\end{enumerate}
\item If $\sum_{i=1}^m(\boldsymbol{p}^{(j)}(i) - \boldsymbol{p}^{(j-1)}(i))^2 < tol\_fun$, set $s^{(j)} = s^{(j-1)}/\rho$. Go to step (8). Else, set $s^{(j)} = s^{(j-1)}$. Go to step (2).
\item If $s^{(j)} \leq \phi$, set $\hat{\mathbf{p}} = \mathbf{p}^{(j)}$. Go to step (9). Else, go to step (2).
\item \textbf{STOP} execution. Set $R = R+1$. Set $\mathbf{z}^{(R)} = \hat{\mathbf{p}}$.  Go to STAGE 2.
\end{enumerate}

\textbf{STAGE : 2}
\begin{enumerate}
\item If $R \leq max\_runs$ and $\mathbf{z}^{(R)} \neq \mathbf{z}^{(R-1)}$, go to step (2). Else $\mathbf{z}^{(R)}$ is the final solution. \textbf{STOP} and \textbf{EXIT}.
\item Set $\rho = \rho_2$ keeping other tuning parameters ($\phi, \lambda$ and $s_{initial}$) fixed. Repeat algorithm described in STAGE 1 setting $\mathbf{p}^{(1)} = \mathbf{z}^{(R)}$.
\end{enumerate}

RMPSS algorithm is a modified version of GPS where \textit{run}s are performed recursively starting with large step-size in order to create enough opportunity to jump out of the local neighbourhood in case the objective function value at the current solution is higher than atleast another local minimum. But just like any other Black-box optimization techniques (e.g., GA, SA, GPS etc) RMPSS may not reach the global minimum on every occasion while minimizing any Black-box function. The performance of Black-box optimization algorithms can be compared in terms of amount of average time spent and proportion of successful convergences to the global minimum while minimizing various range of challenging optimization problems (including popular benchmark functions e.g., Ackley's function, Griewank's function etc) for various dimensions. The comparative studies on wide range of functions considered in Section \ref{sec_app} provides more in-depth knowledge about the performance of RMPSS compared to several other popular optimization techniques.


\section{Theoretical Property}
\label{sec_theory}
In this section it is shown that if the objective function is continuous, differentiable and convex, then starting from any given starting point within the domain, executing single \textit{run} of the RMPSS would yield the solution where the global minimum of the objective function is achieved.
Consider the following theorem.
\begin{theorem}
\label{theorem}
 Suppose $\mathbf{S} = \{(x_1,\cdots,x_n) \in \mathrm{R}^n \; : \; \sum_{i=1}^n x_i = 1, x_i \geq 0, i=1,\cdots,n\}$ and $f$ is convex, continuous and differentiable on $\mathbf{S}$. Consider a sequence $\delta_k = \frac{s}{\rho^k}$ for $k\in \mathrm{N}$ and $s>0, \rho>1$. Suppose $\mathbf{u}$ is a point in $\mathbf{S}$ such that all its coordinates are positive. Define $\mathbf{u}_k^{(i+)} = (u_1-\frac{\delta_k}{n-1},\ldots, u_{i-1}-\frac{\delta_k}{n-1},u_i+\delta_k,u_{i+1}-\frac{\delta_k}{n-1}, \ldots, u_n-\frac{\delta_k}{n-1})$ and $\mathbf{u}_k^{(i-)} = (u_1+\frac{\delta_k}{n-1},\ldots, u_{i-1}+\frac{\delta_k}{n-1},u_i-\delta_k,u_{i+1}+\frac{\delta_k}{n-1}, \ldots, u_n+\frac{\delta_k}{n-1})$ for $i=1,\cdots,n$. If for all $k \in \mathrm{N}$, $f(\mathbf{u}) \leq f(\mathbf{u}_k^{(i+)})$ and $f(\mathbf{u}) \leq f(\mathbf{u}_k^{(i-)})$ (whenever $\mathbf{u}_k^{(i+)}, \mathbf{u}_k^{(i-)}\in \mathbf{S}$) for all $i = 1,\cdots,n$, the global minimum of $f$ occurs at $\mathbf{u}$. 
 \end{theorem}
 
\begin{proof}[Proof of Theorem \ref{theorem}]
Fix some $i \in \{1,\ldots, n\}$. Define
\begin{align*}
r_1 =& \min \{(n-1)u_1, \ldots, (n-1)u_{i-1}, (1-u_1), (n-1)u_{i+1},\ldots, (n-1)u_n\}, \\
r_2 =& \min \{(n-1)(1-u_1), \ldots, (n-1)(1-u_{i-1}), u_1, (n-1)(1-u_{i+1}),\ldots, \\
 & (n-1)(1-u_n)\}.
\end{align*}
Set $r = \min\{r_1,r_2\}$. Since $\delta_k$ is strictly decreasing sequence going to zero, there exist a $N \in \mathbb{Z}$ such that for all $k \geq N$, $\delta_k < r$. Fix some $i \in \{1,\ldots,n\}$. Hence $\mathbf{u}_N^{(i+)}, \mathbf{u}_N^{(i-)} \in \mathbf{S}$.  

Once we fix the first $(n-1)$ coordinates of any element in $\mathbf{S}$, the $n$-th coordinate can be derived by subtracting the sum of the first $(n-1)$ coordinates from $1$. Define $$\mathbf{S}^* = \{(x_1,\cdots,x_{n-1}) \in \mathrm{R}^{n-1} \; : \; \sum_{i=1}^n x_i < 1, x_i \geq 0, i=1,\cdots,n-1\}. $$ Define $\mathbf{u}^*  = (u_1,\ldots,u_{n-1})$ and 
\begin{align*}
\mathbf{u}_k^{*(i+)} & = (u_1-\frac{\delta_k}{n-1},\ldots, u_{i-1}-\frac{\delta_k}{n-1},u_i+\delta_k,u_{i+1}-\frac{\delta_k}{n-1}, \ldots,u_{n-1}-\frac{\delta_k}{n-1}), \\
\mathbf{u}_k^{*(i-)} & = (u_1+\frac{\delta_k}{n-1},\ldots, u_{i-1}+\frac{\delta_k}{n-1},u_i-\delta_k,u_{i+1}+\frac{\delta_k}{n-1}, \ldots,u_{n-1}+\frac{\delta_k}{n-1}),
\end{align*}
for $i = 1,\ldots,n-1$. Note that $\mathbf{u}^*, \mathbf{u}_k^{*(i+)}, \mathbf{u}_k^{*(i-)}$ are the first $(n-1)$ coordinates of $\mathbf{u}, \mathbf{u}_k^{(i+)}, \mathbf{u}_k^{(i-)}$ respectively. Define $f^* : \mathbf{S}^* \mapsto \mathbb{R}$ such that $$f^*(x_1,\ldots,x_{n-1}) = f(x_1,\ldots,x_{n-1}, 1 -\sum_{i=1}^{n-1}x_i).$$ Hence we have $f^*(\mathbf{u}^*) = f(\mathbf{u})$, $f^*(\mathbf{u^*}_k^{(i+)}) = f(\mathbf{u}_k^{(i+)})$ and $f^*(\mathbf{u^*}_k^{(i-)}) = f(\mathbf{u}_k^{(i-)})$. Since, $f$ is continuous and differentiable on $\mathbf{S}$, $f^*$ is continuous and differentiable on $\mathbf{S}^*$. Convexity of $f$ implies $f^*$ is convex on $\mathbf{S^*}$. Consider $\mathbf{x}^*_1, \mathbf{x}^*_2 \in \mathbf{S}^*$. Suppose $\mathbf{x}_1, \mathbf{x}_2 \in \mathbf{S}$ are such that their first $(m-1)$ co-ordinates are same as $\mathbf{x}^*_1$ and $\mathbf{x}^*_2$ respectively. Take any $\gamma \in (0,1)$. Now
\begin{align*}
\gamma f^*(\mathbf{x}^*_1) + (1-\gamma)f^*(\mathbf{x}^*_2) & = \gamma f(\mathbf{x_1}) + (1-\gamma)f(\mathbf{x_2}) \\
& \geq f(\gamma \mathbf{x}_1 + (1-\gamma)\mathbf{x}_2) \\
& = f^*(\gamma \mathbf{x}^*_1 + (1-\gamma)\mathbf{x}^*_2).
\end{align*}
Hence $f^*$ is also convex. Define $h_i : U_i \mapsto \mathbf{S}^*$ such that $$h_i(z) = (u_1 - \frac{z}{n-1},\ldots,u_{i-1} - \frac{z}{n-1}, u_i + z,u_{i+1}- \frac{z}{n-1}, \ldots u_{n-1}- \frac{z}{n-1})$$ for  $i=1,\ldots, n-1$, where $U_i = [-\delta_N,\delta_N]$ (since each co-ordinate of $\mathbf{u}$ is positive, $\mathbf{u}^* \in \mathbf{S}^*$. Note that the way $N$ is chosen ensures $h_i(U_i) \subset \mathbf{S}^*$). Define $g_i :  U_i \mapsto \mathrm{R}$ for $i=1,\ldots,n-1$ such that $g_i = f^* \circ h_i$. Hence we have $$g_i(z) = f^*(u_1 - \frac{z}{n-1},\ldots,u_{i-1} - \frac{z}{n-1}, u_i + z,u_{i+1}- \frac{z}{n-1}, \ldots u_{n-1}- \frac{z}{n-1})$$ for $i=1,\ldots, n-1$.

It is noted that $h_i$ is continuous on $U_i=[-\delta_N, \delta_N]$ and differentiable on $(-\delta_N, \delta_N)$ for $i=1,\ldots,n-1$ and $f^*$ is continuous and differentiable on $\mathbf{S}^*$. Composition of two continuous functions is continuous and the composition of two differentiable functions is differentiable. Hence, $g_i$ is continuous on $U_i=[-\delta_N, \delta_N]$ and differentiable on $(-\delta_N, \delta_N)$. 

Take any $i \in \{1,\ldots,n-1\}$. Note that $g_i(\delta_N) = f^*(\mathbf{u}_N^{*(i+)}),\, g_i(-\delta_N) = f^*(\mathbf{u}_N^{*(i-)})$ and $g_i(0) = f^*(\mathbf{u^*})$.  So, $g_i(0) \leq g_i(-\delta_N)$ and $g_i(0) \leq g_i(\delta_N)$. Without loss of generality, assume $f^*(\mathbf{u}_N^{*(i-)}) \leq f^*(\mathbf{u}_N^{*(i+)})$ which implies $g_i(0) \leq g_i(-\delta_N) \leq g_i(\delta_N)$. 

Since we have $g_i(0) \leq g_i(-\delta_N) \leq g_i(\delta_N)$, from continuity of $g_i$ we can say that there exists $w\in [0,\delta_N]$ such that $g_i(w) = g_i(-\delta_N) \geq g_i(0)$. Now since $g_i$ is continuous on $[-\delta_N, \delta_N]$ and differentiable on $(-\delta_N, \delta_N)$, it implies $g_i$ is continuous on $[-\delta_N, w]$ and differentiable on $(-\delta_N, w)$. Using Mean value theorem we can say that there exists a point $v \in [-\delta_N, w]$ such that $g_i^{\prime}(v) = 0$.

We claim that $g_i^{\prime}(v) = 0$ holds for $v = 0$. Suppose $g_i^{\prime}(0) \neq 0$. Assume $g_i^{\prime}(v^*) = 0$  for some $v^* \neq 0$ and $v^* \in (-\delta_N,w)$. Without loss of generality, we assume $v^* > 0$. Since $h_i$ and $f$ are convex, $g_i$ is also convex. Now $g_i^{\prime}(v^*) = 0$ implies $v^*$ is a local minima. On the other hand, since $g_i^{\prime}(0) \neq 0$, implies $0$ is not a critical point or local minima. Hence, $g_i(0) > g_i(v^*)$. Take $N_1 \in \mathbb{Z}$ such that $0 <\delta_{N_1} < v^*$. Hence there exists a $\lambda \in (0,1)$ such that $\delta_{N_1} = \lambda. 0 + (1-\lambda). v^*$. Now,
\begin{align*}
g_i(\delta_{N_1}) &= g_i(\lambda. 0 +(1-\lambda).v^*) \\
& \leq \lambda g_i(0) + (1-\lambda)g_i(v^*) \\
& = g_i(0) + (1-\lambda)(g_i(v^*) - g_i(0)) \\
& = g_i(0) - (1-\lambda)(g_i(0) - g_i(v^*)) \\
& < g_i(0).
\end{align*}
But, we know for all $k \in \mathbb{Z}$, $g_i(0) \leq g_i(\delta_k)$ which implies $g_i(0) \leq g_i(\delta_{N_1})$. It is a contradiction. \\
Hence we have $g_i^{\prime}(0) = 0$. Now 
\begin{align*}
g_i^{\prime}(0) &= \bigg[\frac{\partial}{\partial \epsilon}g_i(\epsilon)\bigg]_{\epsilon = 0} \\ 
&= \bigg[\frac{\partial}{\partial \epsilon}f^*(h_i(\epsilon))\bigg]_{\epsilon = 0} \\ 
& = \bigg[\frac{\partial}{\partial h_i(\epsilon)}f^*(h_i(\epsilon))\bigg]_{\epsilon = 0} \bigg[\frac{\partial}{\partial \epsilon}h_i(\epsilon)\bigg]_{\epsilon = 0}.
\end{align*}
Now $h_i(0) = \mathbf{u}^*$. Hence 
\begin{align*}
\bigg[\frac{\partial}{\partial h_i(\epsilon)}f^*(h_i(\epsilon))\bigg]_{\epsilon = 0} &= \nabla f^*(\mathbf{u}^*) \\ 
&=\bigg[\frac{\partial}{\partial x_1}f^*(\mathbf{u}^*), \ldots, \frac{\partial}{\partial x_{n-1}}f^*(\mathbf{u}^*) \bigg] \\
&=\bigg[\nabla_1, \ldots, \nabla_{n-1} \bigg]
\end{align*}
where $\nabla_i = \frac{\partial}{\partial x_{i}}f^*(\mathbf{u}^*)$ for $i=1,\ldots,n-1$. and 
\begin{align*}
\frac{\partial}{\partial \epsilon}h_i(\epsilon) & = [a_{i1},\ldots, a_{i(n-1)}]^T
\end{align*}
where $a_{ii} = 1$ and $a_{ij} = -\frac{1}{n-1}$ for $j \in \{1,\ldots,n-1\} \setminus \{i\}$
Hence 
\begin{align*}
\bigg[\frac{\partial}{\partial \epsilon}g_i(\epsilon)\bigg]_{\epsilon = 0} &= \bigg[\nabla_1, \ldots, \nabla_{n-1} \bigg] \bigg[a_{i1},\ldots, a_{i(n-1)}\bigg]^T \\
& = \bigg[a_{i1},\ldots, a_{i(n-1)}\bigg] \begin{bmatrix}
    \nabla_1  \\
    \vdots \\
    \nabla_{n-1} 
  \end{bmatrix} \\
  & = 0.
\end{align*}
Since this equation holds for all $i= 1,\cdots,n-1$, we have $A\mathbf{x} = \mathbf{0}$ where
\begin{align*}
A_{n \times n} = \begin{bmatrix}
    1 & -\frac{1}{n-1} & \cdots & -\frac{1}{n-1}  \\
    -\frac{1}{n-1} & 1 &\cdots & -\frac{1}{n-1}\\
    \vdots & \vdots & \ddots & \vdots \\
    -\frac{1}{n-1} & -\frac{1}{n-1} & \cdots & 1
  \end{bmatrix}, \; \mathbf{x}_{n \times 1} = \begin{bmatrix}
    \nabla_1  \\
    \vdots \\
    \nabla_{n-1} 
  \end{bmatrix}.
\end{align*}
\\
Since $A$ is full rank for $n \in \mathbb{N} \setminus \{1\}$, $A\mathbf{x} = \mathbf{0}$ implies $\mathbf{x} = \mathbf{0}$. Hence $\frac{\partial}{\partial x_{i}}f^*(\mathbf{u}^*)=0$ for all $i = 1,\ldots, n-1$. Hence $\mathbf{u}^*$ is a critical point. Since $f^*$ is convex, a local minima occurs at $\mathbf{u}^*$. But for a convex function, global minimum occurs at any local minimum. Hence global minimum of $f^*$ occurs at $\mathbf{u}^*$, which clearly implies global minimum of $f$ occurs at $\mathbf{u}$.
\end{proof}

Suppose the solution given by RMPSS is a point $\mathbf{u} \in \mathbf{S}$ such that all of coordinate elements are greater than zero. Now, RMPSS returns the final solution to the user when two consecutive \textit{run}s return the same solution. It implies in the last \textit{run}, for all movements of step sizes $\delta_k=\frac{s_{initial}}{\rho^k}$ (until $\delta_k$ gets smaller than the \textit{step size threshold}) the objective function value is checked at $\mathbf{u}_k^{(i+)},\mathbf{u}_k^{(i-)}$ and $f(\mathbf{u}) \leq f(\mathbf{u}_k^{(i+)})$ and $f(\mathbf{u}) \leq f(\mathbf{u}_k^{(i-)})$ hold for all $i=1,\ldots,n$. So making the value of \textit{step size threshold} sufficiently small, RMPSS would reach the global minimum under assumed regularity conditions of the objective function. Note that, for this convergence results to be hold true, the value of $\lambda$ should be taken to be zero. But, in practical, it is noted that setting a small non-zero value of $\lambda$ (specially in high-dimensional problems with possibility of sparsity) increases the efficiency and accuracy of the solution provided by this algorithm.
\section{Generalization to some other cases}
\label{general}
In this section we describe how RMPSS can be used to optimize any objective function on two well-known constrained spaces, namely simplex inequality constraint and linearly constraint (with positive coefficients).
\subsection{Simplex Inequality}
\label{sec_ineq}
Consider the case where the optimization problem is given by 
\begin{align}
minimize \;:& \; f(p_1,\cdots,p_m) \nonumber \\
subject \; to \; :& \; p_i \geq 0, \; 1 \leq i \leq m, \; \sum_{i=1}^mp_i \leq 1. 
\label{ineq_problem}
\end{align}
Under this scenario, a slack variable $p_{m+1}$ is introduced such that $p_{m+1} \geq 0$ and $\sum_{i=1}^{m+1}p_i = 1$. Define $f_1(p_1,\ldots,p_{m+1}) = f(p_1,\ldots,p_m)$. So, the modified optimization problem which is equivalent to Equation (\ref{ineq_problem}) is given by
\begin{align}
minimize \;:& \; f_1(p_1,\cdots,p_{m+1}) \nonumber \\
subject \; to \; :& \; p_i \geq 0, \; 1 \leq i \leq m+1, \; \sum_{i=1}^{m+1}p_i = 1,
\label{ineq_problem_sol}
\end{align}
which can be easily solved using the proposed algorithm.
\subsection{Linear Constraint with Positive Coefficients}
\label{sec_lin}
Now, consider the case where the optimization problem is given by 
\begin{align}
minimize \;:& \; f(x_1,\cdots,x_m) \nonumber \\
subject \; to \; :& \; x_i \geq 0,  \; 1 \leq i \leq m, \; \sum_{i=1}^ma_ix_i = K, a_i>0,K>0.
\label{lin_problem}
\end{align}
 To solve this problem, consider the change of variable given by $y_i = \frac{a_ix_i}{K}$ for $i=1,\ldots,m$. $y_i$ is non-negative since $K > 0, x_i \geq 0$ and $a_i > 0$ for $i=1,\ldots,m$. Now, the constraint $\sum_{i=1}^ma_ix_i = K$ is equivalent to $\sum_{i=1}^my_i = 1$ after re-parametrization. Consider the mapping $g : \mathbb{R}^m \mapsto \mathbb{R}^m$   $$g(y_1,\ldots,y_m) = \bigg(\frac{Ky_1}{a_1},\ldots,\frac{Ky_m}{a_m}\bigg).$$ Define $h : \mathbb{R}^m \mapsto \mathbb{R}$ such that $h = f\circ g$. So, 
\begin{align*}
h(y_1,\ldots,y_m) &= f(g(y_1,\ldots,y_m)) \\
&= f\bigg(\frac{Ky_1}{a_1},\ldots,\frac{Ky_m}{a_m}\bigg) \\
&= f(x_1,\ldots,x_m).
\end{align*}
Hence, the optimization problem in Equation (\ref{lin_problem}) is equivalent to
\begin{align}
minimize \;:& \; h(y_1,\cdots,y_m) \nonumber \\
subject \; to \; :& \; y_i \geq 0,  \; 1 \leq i \leq m, \; \sum_{i=1}^my_i = 1,
\label{lin_problem_sol}
\end{align}
which can be solved using the proposed algorithm.

\section{Application to non-convex global optimization on Simplex }
\label{sec_app}
In this section, we compare the performance of the proposed method (RMPSS) with  three constrained optimization methods : the `interior-point' (IP) algorithm, `sequential quadratic programming' (SQP) and `genetic algorithm' (GA) based on optimization of challenging objective functions on simplex constrained parameter space.  All of the above-mentioned algorithms are available in Matlab R2016b (The Mathworks) via the Optimization Toolbox functions \textit{fmincon} (for IP and SQP algorithm)  and \textit{ga} (for GA). IP and SQP algorithms return a local minimum as the solution but they are less time consuming, in general. On the other hand GA is a heuristic algorithm which attempts to find global minimum, and it is more time consuming. For the following comparative studies, we consider the algorithm is successful in reaching the global minimum if the absolute difference of function value at the solution (returned by the algorithm) and the true minimum value of the objective function is less than $10^{-2}$. For maximization problems, the objective function is taken to be the negative of the true function which needs to be maximized and same convergence criteria is considered as mentioned previously. For RMPSS algorithm, the values of all the tuning parameters are taken to be same as mentioned in Section \ref{sec_algo}. RMPSS algorithm is implemented in Matlab R2016b. The code for RMPSS algorithm is made available at the following link\footnote{\url{https://github.com/priyamdas2/RMPSS}}. For IP and SQP algorithms, the upper bound for maximum number of  iterations and function evaluations is set to be infinity. For GA, we use the default options of \textit{ga} function in Matlab R2016b.  We perform the simulations in a machine with 64-Bit Windows 8.1, Intel i7 3.60GHz processors and 32GB RAM.
\subsection{Maximum of two Gaussian densities}
\label{sect:gauss}
Consider the problem
 \begin{align}
 maximize \;:& \; max\big(8*\phi( \mathbf{p};\mu_1, \Sigma_1), \; 5*\phi(\mathbf{p};\mu_2, \Sigma_2) \big)  \nonumber \\
subject \; to \; :& \; p_1,p_2 \geq 0, \quad p_1+p_2 = 1,
\label{eq:example_normal}
\end{align}
where $\mathbf{p} = (p_1,p_2)$ is our parameter of interest, $\phi(\mathbf{x};\mu,\Sigma)$ denotes the normal density at $\mathbf{x}$ with mean $\mu$ and covariance matrix $\Sigma$. Here, $\mu_1 = \begin{bmatrix}
    0.25 \\
    0.75
\end{bmatrix}, \mu_2 = \begin{bmatrix}
    0.8 \\
    0.2
\end{bmatrix}$ and $\Sigma_1 = \Sigma_2 = \begin{bmatrix}
    0.1 & 0\\
    0 & 0.1
\end{bmatrix}$. In Figure \ref{fig:02}, the function is plotted on the restricted parameter space which is of our interest. Note that the function has two local maximums at $(p_1,p_2) = (0.8,0.2)$ and $(0.25,0.75)$, the latter being the global maximum. For comparative study, taking $(p_1,p_2) = (0.8,0.2)$ (which is a local maximum, not the global maximum) as the starting point, RMPSS, SQP, IP and GA methods are used to find the global maximum. It is observed that only RMPSS and GA reach the global maxima $(0.25,0.75)$ successfully, while IP and SQP get stuck at the starting point as that is a local maximum. We also consider another study where this objective function is maximized using all above-mentioned algorithms starting from 100 randomly generated starting point within the simplex constrained space. Only RMPSS and GA reach the global maximum at every occasion. Note that more than 21 folds improvement in average computation time is obtained using RMPSS over GA (see Table \ref{table:example_123}). 
 \begin{figure}
     \subfigure[]{\includegraphics[width=0.5\textwidth]{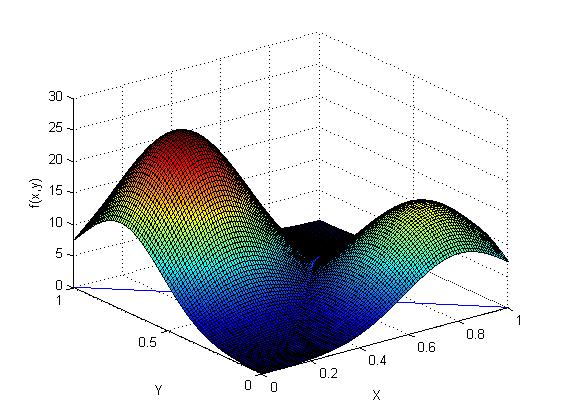}\label{fig:01} }
     \subfigure[]{\includegraphics[width=0.5\textwidth]{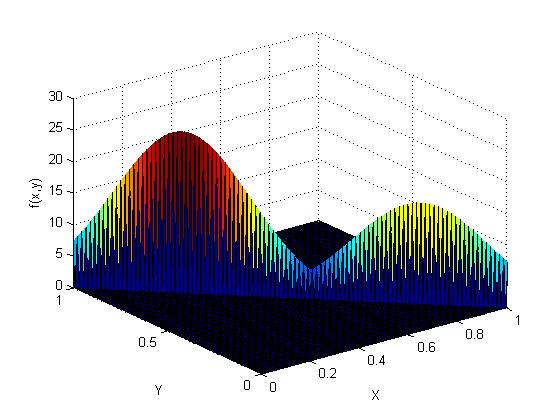}\label{fig:02} }
       \caption{In Example (\ref{sect:gauss}) (a) Value of $f(\mathbf{p})$ on unrestricted $X-Y$ plane. (b) $f(\mathbf{p})$ on restricted 1-simplex defined by  $x+y=1; \; x,y \geq 0$.}
       \label{figure_0} 
 \end{figure}

\subsection{Modified Easom function on simplex}
\label{sect:easom}
Consider the following problem
 \begin{align}
 maximize \;:& \; \cos(6\pi p_1)\cos(6\pi p_2)\cos(6\pi p_3)\exp(-\sum_{i=1}^3(3\pi p_i-\pi)^2)  \nonumber \\
subject \; to \; :& \; p_1,p_2,p_3 \geq 0, \;p_1+p_2+p_3 = 1.
\label{eq:example_easom}
\end{align}
This function has multiple local maxima (see Figure \ref{fig:example_easom}) with the global maxima at $\mathbf{p}=(\frac{1}{3},\frac{1}{3},\frac{1}{3})$, the functional value at this point being 1.  Starting from 100 randomly generated starting points within the simplex constrained space, RMPSS, SQP, IP and GA are used to find the global maximum. The comparative performance of all above-mentioned algorithms are noted in Table \ref{table:example_123}. In this scenario also only RMPSS and GA converge successfully to true global minimum for all starting points. It is noted that around 85 fold time improvement is obtained using RMPSS over GA.
 \begin{figure*}[] 
    \centering
    \includegraphics[width=0.7\linewidth]{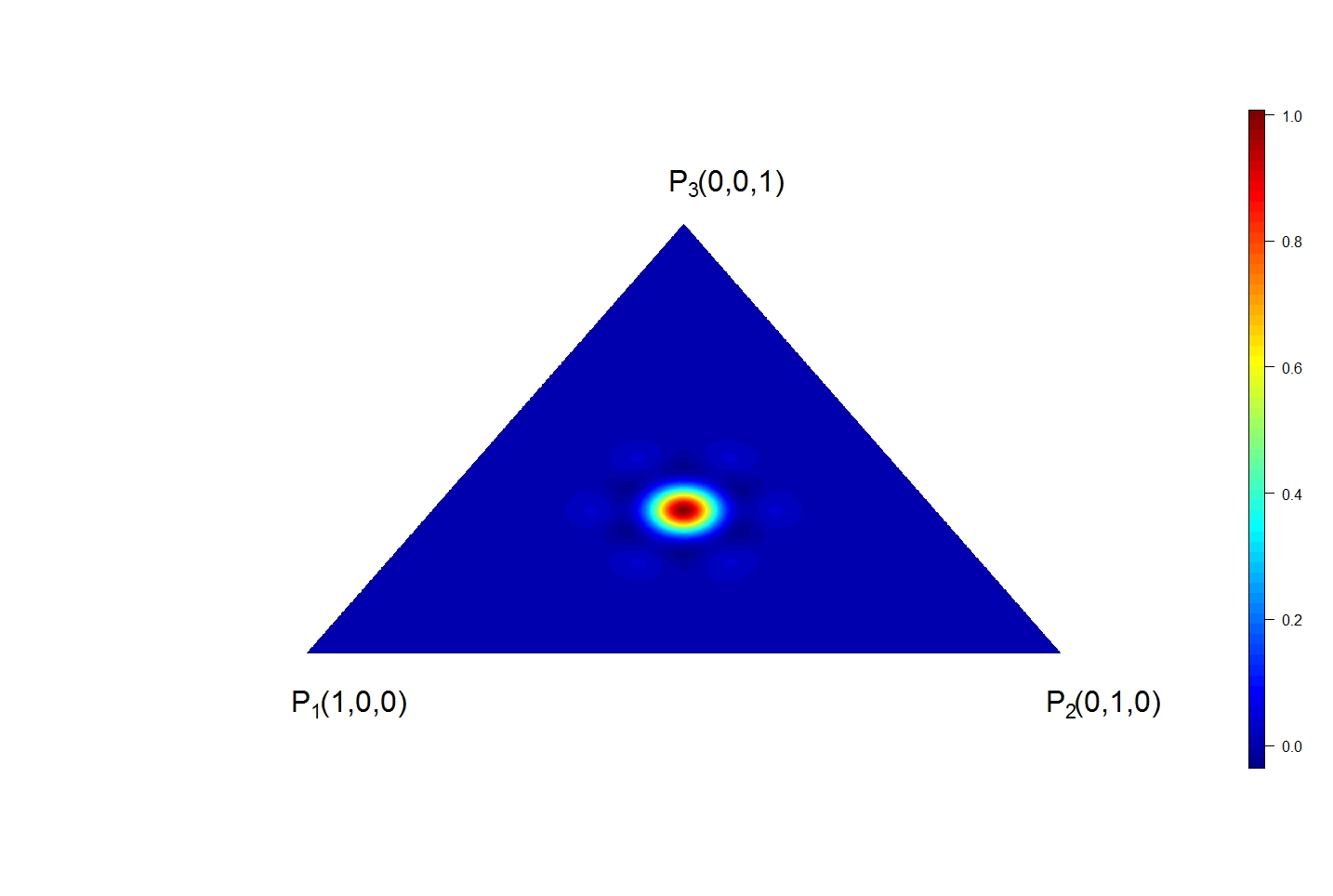} 
    \caption{Heat-map of functional values of modified Easom function on 3-simplex in Example (\ref{sect:easom}).}
      \label{fig:example_easom} 
\end{figure*} 

\subsection{Non-linear non-concave maximization on 2-simplex}
 \label{sect:sin}
 Here we consider a problem of maximizing a non-linear non-concave function (which is equivalent to minimizing non-linear non-convex function) on the a linearly constrained space on $\mathrm{R}^2$. 
 \begin{align}
maximize \;:& \; sin(\frac{7\pi x}{4}) + sin(\frac{7\pi y}{4}) - 2(x-y)^2  \nonumber \\
subject \; to \; :& \; 3x+2y \leq 6,\quad x,y \geq 0 
\label{eq:example_sin_cos}
\end{align}
In Figure \ref{fig:example_sin_cos}, a heat-map of the values of this function over the parameter space is plotted. This function has 4 local maximums out of which the global maximum is located at $(x,y)=(\frac{2}{7},\frac{2}{7})=(0.2857,0.2857)$, the value of the objective function value being 2 at this point. Note that any point in the feasible region (defined in \eqref{eq:example_sin_cos}) is in the convex hull generated by $(0,0), (2,0)$ and $(0,3)$ on $\mathbb{R}^2$. Now, for any given point within a triangle, there exist unique non-negative weights $\{p_1,p_2,p_3 :p_i\geq 0, \sum_{i}p_i = 1\}$ such that the coordinate of that point can be given by the weighted average of the coordinates of the vertices. Thus, once the weight vector $(p_1,p_2,p_3)$ is estimated for which the the objective function is maximized, the solution in terms of $(x,y)$ can be calculated. In Table \ref{table:example_123}, it is noted that RMPSS outperforms the other algorithms based on the comparative study of the number of successful convergences for all methods starting from 100 randomly generated starting points, and using RMPSS around 66 folds improvement in computation time is observed over GA.
 \begin{figure*}
     \subfigure[]{\includegraphics[width=0.5\textwidth]{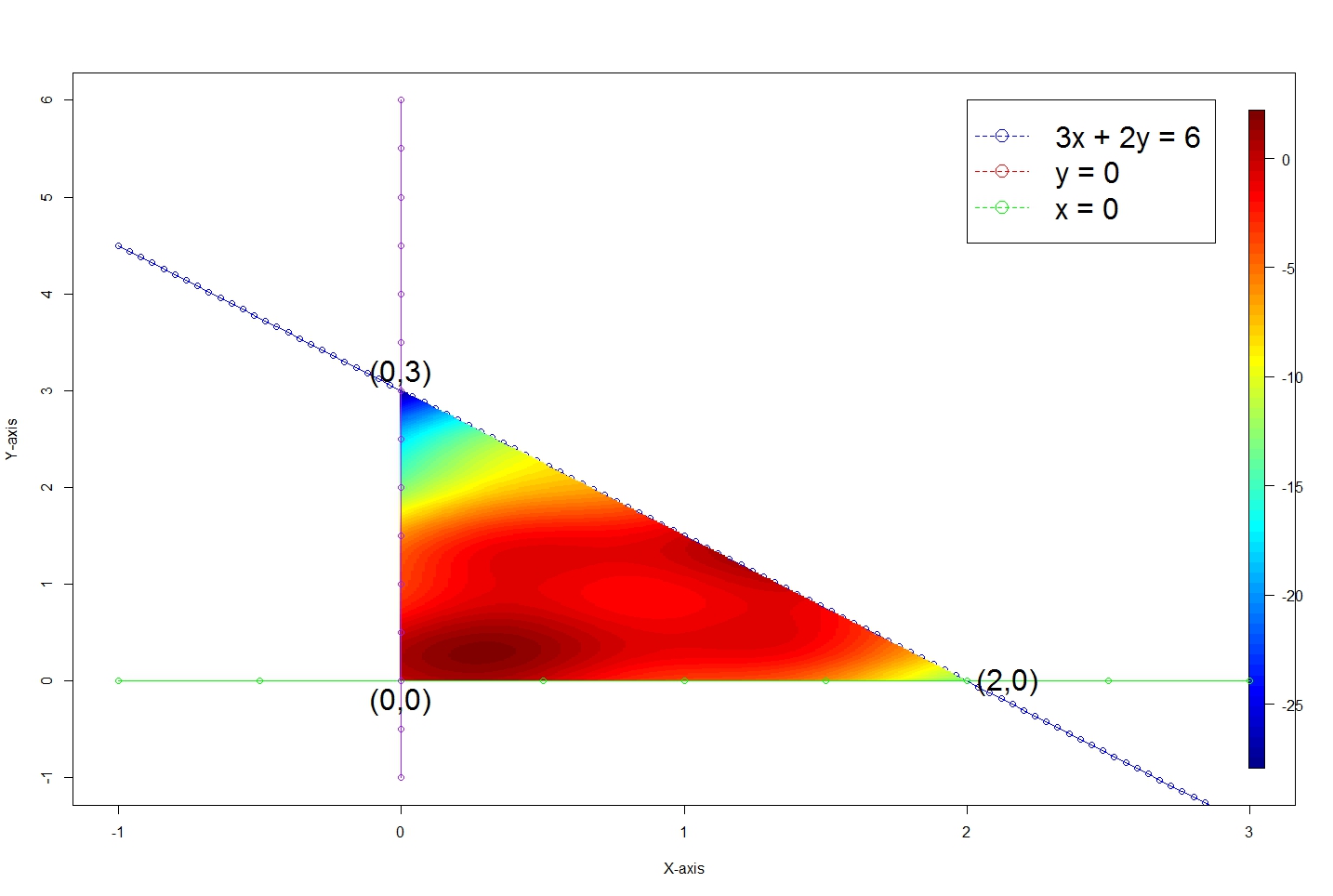}\label{fig:example_sin_cos} }
     \subfigure[]{\includegraphics[width=0.5\textwidth]{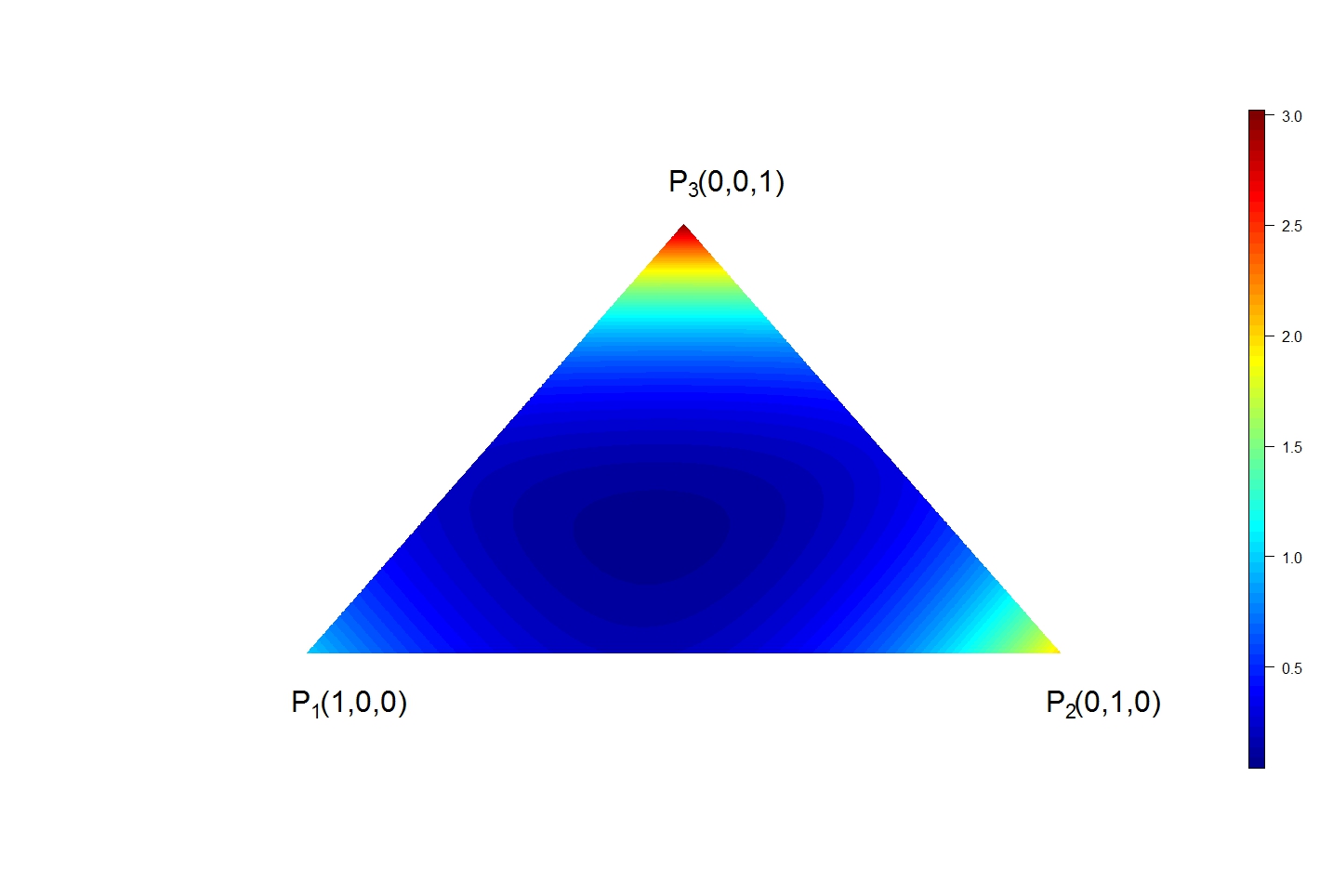}\label{fig:example_power_4} }
       \caption{(a) Heat-map of $f(\mathbf{p})$ on restricted $X-Y$ plane in Example (\ref{sect:sin}). (b) Heat-map of $f(\mathbf{p})$ for $n=3$ in Example (\ref{sect:border}).}
 \end{figure*}
 
  \begin{table}[]
  \centering
  \resizebox{\columnwidth}{!}{%
  \begin{tabular}{|l|c|c|c|c|c|c|}
  \hline
  \multicolumn{1}{|c|}{\multirow{2}{*}{Algorithms}} & \multicolumn{2}{c|}{Example \ref{sect:gauss}} & \multicolumn{2}{c|}{Example \ref{sect:easom}} & \multicolumn{2}{c|}{Example \ref{sect:sin}} \\ \cline{2-7} 
  \multicolumn{1}{|c|}{} & \begin{tabular}[c]{@{}c@{}}Success \\ (\%)\end{tabular} & \begin{tabular}[c]{@{}c@{}}Avg. Time\\ (sec)\end{tabular} & \begin{tabular}[c]{@{}c@{}}Success\\  (\%)\end{tabular} & \begin{tabular}[c]{@{}c@{}}Avg. Time\\ (sec)\end{tabular} & \begin{tabular}[c]{@{}c@{}}Success\\ (\%)\end{tabular} & \begin{tabular}[c]{@{}c@{}}Avg. Time\\ (sec)\end{tabular} \\ \hline
  RMPSS & 100 & 0.071 & 100 & 0.016 & 100 & 0.019 \\ \hline
  SQP & 67 & 0.014 & 31 & 0.014 & 44 & 0.025 \\ \hline
  IP & 71 & 0.023 & 47 & 0.030 & 47 & 0.014 \\ \hline
  GA & 100 & 1.494 & 100 & 1.354 & 100 & 1.270 \\ \hline
  \end{tabular}}
   \caption{Comparison of required time and number of successful convergence for solving the problems in Example (\ref{sect:gauss}), (\ref{sect:easom}) and (\ref{sect:sin}) using  RMPSS, SQP, IP and GA starting from 100 randomly generated points.}
   \label{table:example_123}
  \end{table}
 \subsection{Optimization of function with multiple local maximums on boundary points for various dimensions}
  \label{sect:border}
Consider the problem 
 \begin{align}
 maximize \;:& \; \sum_{i=1}^nip_i^4  \nonumber \\
subject \; to \; :& \; p_i \geq 0,\; i=1,\cdots,n \; \quad \sum_{i=1}^np_i = 1,
\label{eq:example_power_4}
\end{align}
where $n$ is any positive integer. Note that each boundary point of the simplex constrained space is a local maximum for this function. But the global maximum occurs at $\hat{\mathbf{p}}=(p_1,\cdots,p_n)$ where $p_1=\ldots=p_{n-1}=0$ and $p_n=1$ attaining global maximum 1. With increasing value of $n$, it gets harder to estimate the global maximum. In Figure \ref{fig:example_power_4}, we plot the heat map of $f(\mathbf{p})$ for $n=3$. It can be seen that this function has three local maxima at $P_1=(1,0,0), P_2=(0,1,0)$ and $P_3=(0,0,1)$ out of which $P_3 = (0,0,1)$ is the global maxima. Starting from 100 randomly generated points within simplex constrained space, we perform a comparative study of performances of all the above-mentioned algorithms for different possible values of $n=5,10,25,50,100$. In Table (\ref{table:example_power_4}), it is noted that unlike GA, RMPSS works well for high-dimensional cases as well. It is also noted that the required computation time for RMPSS algorithm increases approximately at a linear rate with dimension of the problem.
\begin{table}[]
\centering
\resizebox{\columnwidth}{!}{%
\begin{tabular}{|l|c|c|c|c|c|c|c|c|c|c|}
\hline
\multirow{2}{*}{Algorithms} & \multicolumn{2}{c|}{n=5} & \multicolumn{2}{c|}{n=10} & \multicolumn{2}{c|}{n=25} & \multicolumn{2}{c|}{n=50} & \multicolumn{2}{c|}{n=100} \\ \cline{2-11} 
 & \begin{tabular}[c]{@{}c@{}}No. of\\ success\end{tabular} & \begin{tabular}[c]{@{}c@{}}Avg.\\ time\end{tabular} & \begin{tabular}[c]{@{}c@{}}No. of\\ success\end{tabular} & \begin{tabular}[c]{@{}c@{}}Avg.\\ time\end{tabular} & \begin{tabular}[c]{@{}c@{}}No. of\\ success\end{tabular} & \begin{tabular}[c]{@{}c@{}}Avg.\\ time\end{tabular} & \begin{tabular}[c]{@{}c@{}}No. of\\ success\end{tabular} & \begin{tabular}[c]{@{}c@{}}Avg.\\ time\end{tabular} & \begin{tabular}[c]{@{}c@{}}No. of\\ success\end{tabular} & \begin{tabular}[c]{@{}c@{}}Avg.\\ time\end{tabular} \\ \hline
RMPSS & 100 & 0.040 & 100 & 0.079 & 100 & 0.226 & 100 & 0.394 & 100 & 0.909 \\ \hline
SQP & 31 & 0.008 & 22 & 0.010 & 16 & 0.017 & 15 & 0.028 & 7 & 0.063 \\ \hline
IP & 30 & 0.022 & 22 & 0.026 & 15 & 0.052 & 22 & 0.094 & 26 & 0.183 \\ \hline
GA & 4 & 2.910 & 0 & 55.762 & 0 & 51.091 & 0 & 50.345 & 0 & 53.232 \\ \hline
\end{tabular}
}
\caption{Comparison of required time and number of successful convergence for solving the problem in Example (\ref{sect:border}) using  RMPSS, SQP, IP and GA for $n=5,10,25,50,100$ starting from 100 randomly generated points in each case.}
\label{table:example_power_4}
\end{table}
 \subsection{Transformed Ackley's Function on Simplex}
 \label{sect:ackley}
 Consider a function $f$ needs to be minimized on a $d$-dimensional hypercube $D^d$ where $D=[l,u]$ for some constants $l,u$ in $\mathbb{R}$. Consider the bijection map $g : D \mapsto [0,\frac{1}{d}]$ such that $g(x_i) = y_i=\frac{x_i-l}{d(u-l)}$ for $i=1,\ldots,d$. Replacing the original parameters of the problem with the transformed parameters we get
   \begin{align*}
   f(x_1,\ldots,x_d) = f(g^{-1}(y_1), \ldots, g^{-1}(y_d)).
   \end{align*}
   Now, define $h : [0,\frac{1}{d}]^d \mapsto \mathbb{R}$ such that $$h(\mathbf{y})=h(y_1,\ldots,y_d)=f(g^{-1}(y_1), \ldots, g^{-1}(y_d)).$$ 
    Consider the set $S = \{(z_1,\ldots,z_d) \; | z_i \geq 0, \; \sum_{i=1}^dz_i \leq 1 \}$. Clearly $[0,\frac{1}{d}]^{d} \subset S$. Define $h^{\prime} : S \mapsto \mathbb{R}$ which is equal to function $h$ considered on the extended domain $S$. We have $y_i \in [0,\frac{1}{d}]$ for $i=1,\ldots,d$ and $0 \leq \sum_{i=1}^dy_i \leq 1$. Define $y_{d+1} = 1 - \sum_{i=1}^dy_i$. Hence $0 \leq y_{d+1} \leq 1$ and $\sum_{i=1}^{d+1}y_i = 1$. So we can conclude that $\bar{\mathbf{y}} = [\mathbf{y}, \; y_{d+1}]\in \Delta^d$ where $\mathbf{y} = (y_1,\ldots,y_{d})$ and 
 \begin{align*}
 \Delta^{d} = \{(y_{1}, \ldots, y_{d+1}) \in \mathbb{R}^{d+1} \; | \; y_i \geq 0, \;  i=1,\ldots,d+1, \; \sum_{i=1}^{d+1}y_i= 1\}.
 \end{align*}
 Now define $\bar{h} : \Delta^d \mapsto \mathbb{R}$ such that $ \bar{h}(\bar{\mathbf{y}}) = \bar{h}(y_1,\ldots,y_{d+1}) = h^{\prime}(y_1,\ldots, y_d)$ for $\bar{\mathbf{y}} \in \Delta^{d}$. Note that $\bar{\mathbf{y}} \in \Delta^{d}$ implies $(y_1,\ldots, y_d) \in S$. Suppose the global minimum of the function $f$ occurs at $(m_1,\ldots,m_d)$ in $D^d$. Hence, the function $\bar{h}$ has the global minimum at $\bar{\mathbf{y}} = \big(g^{-1}(m_1),\ldots, g^{-1}(m_d), 1 - \sum_{i=1}^d g^{-1}(m_i) \big)$ in $\Delta^d$. \\
 
 $d$-dimensional Ackley's function is given by
 \begin{align*}
f(x_1,\ldots,x_d) = -20 \exp(-0.2\sqrt{0.5\sum_{i=1}^dx_i^2}) - \exp(0.5\sum_{i=1}^d\cos(2\pi x_i))+e+20.
 \end{align*}
 The domain of $\mathbf{x}=(x_1,\ldots,x_d)$ is generally taken to be $[-5,5]^d$. The global minimum of $f$ is $0$ (even if considered on $\mathbb{R}^d$) which occurs at $\mathbf{x}^*=(x_1^*,\ldots,x_d^*)=(0,\ldots,0)$. After doing the above mentioned transformations taking $l=-5$ and $u=5$, we get the transformed Ackley's function on a $d$-dimensional unit-simplex $\Delta^d$ given by
 \begin{align*}
 \bar{h}(\bar{\mathbf{y}})= & \bar{h}(y_1,\ldots,y_{d+1}) \\
 = & -20 \exp(-0.2\sqrt{0.5\sum_{i=1}^d(g^{-1}(y_i))^2} - \exp(0.5\sum_{i=1}^d\cos(2\pi g^{-1}(y_i)))+e+20.
 \end{align*}
The global minimum of the transformed Ackley's function on simplex occurs at $\bar{\mathbf{y}}_{1 \times (d+1)}^*=(\frac{1}{2d},\ldots,\frac{1}{2d}, \frac{1}{2})$ which is found using inverse transformation on $\mathbf{x}^*$ as mentioned above. For comparative study we considered RMPSS algorithm for three set of parameter values; the first setup being the default parameter values (as mentioned in Section \ref{sec_algo}) and the other two sets of parameter values being RMPSS (pl1) (`pl' stands for precision level) and RMPSS (pl2). In RMPSS (pl1), we take $\lambda=\phi=10^{-5}$ and for RMPSS (pl2), we take $\lambda=\phi=10^{-7}$ keeping the values of the other parameters same as in the default setup. The main motive behind considering different sets of parameters is to find how the computation performance and time of RMPSS vary as the we increase the precision of the solution. For each algorithm, transformed Ackley's function is optimized for $d=5,10,25,50$ and $100$. In each case, the objective function is minimized starting from 100 randomly chosen points. In Table 3, the average computation time and the minimum value achieved by each algorithm are noted. It is noted that for this function, RMPSS algorithm outperforms all other algorithms by a large margin both in terms of proportion of successful convergences and computation time. It is also observed that taking smaller values of $\lambda$ and $\phi$ improves the accuracy of the solution at the cost of higher computation time. Note that for $d=5$, RMPSS (pl1) yields better solution than GA with a 235 folds improvement in average computation time.
 \subsection{Transformed Griewank's Function on Simplex}
  \label{sect:grie}
 $d$-dimensional Griewank's function is given by
 \begin{align*}
f(x_1,\ldots,x_d) = \frac{1}{4000}\sum_{i=1}^dx_i^2 - \prod_{i=1}^d\cos\big(\frac{x_i}{\sqrt{i}}\big) + 1.
 \end{align*}
 The domain of $\mathbf{x}=(x_1,\ldots,x_d)$ is generally taken to be $[-500,500]^d$. The global minimum of $f$ is $0$ (even if considered on $\mathbb{R}^d$) which occurs at $\mathbf{x}^*=(x_1^*,\ldots,x_d^*)=(0,\ldots,0)$. Similar to the previous problem, after performing the previously mentioned transformations taking $l=-500$ and $u=500$, we get the transformed Griewank's function on a $d$-dimensional unit-simplex $\Delta^d$ given by
 \begin{align*}
 \bar{h}(\bar{\mathbf{y}})= & \bar{h}(y_1,\ldots,y_{d+1}) \\
 = & \frac{1}{4000}\sum_{i=1}^d(g^{-1}(y_i))^2 - \prod_{i=1}^d\cos\big(\frac{g^{-1}(y_i)}{\sqrt{i}}\big) + 1.
 \end{align*}
Similar to the previous scenario, in this case also the global minimum occurs at $\bar{\mathbf{y}}_{1 \times (d+1)}^*=(\frac{1}{2d},\ldots,\frac{1}{2d}, \frac{1}{2})$. We perform the comparative performance study for this function under the similar setup of Example (\ref{sect:ackley}). In this case, RMPSS (pl1 \& pl2) outperforms other methods for high-dimensional cases, while SQP and IP perform relatively better than RMPSS (default, pl1 \& pl2) and GA for smaller dimensional cases. In this case also, the solution improves taking the values of parameters $\lambda$ and $\phi$ smaller in RMPSS. Note that for $d=5$, more than 338 folds improvement in average computation time is noted using RMPSS(pl1) over GA along with more accuracy of the solution.
 \subsection{Transformed Rastrigin's Function on Simplex}
  \label{sect:rast}
 $d$-dimensional Rastrigin's function is given by
  \begin{align*}
 f(x_1,\ldots,x_d) = 10d+\sum_{i=1}^d[x_i^2-10\cos(2\pi x_i)]
  \end{align*}
  The domain of $\mathbf{x}=(x_1,\ldots,x_d)$ is generally taken to be $[-5,5]^d$. After transformation in the above-mentioned way, the transformed Rastrigin's function on $\Delta^d$ is given by
 \begin{align*}
  \bar{h}(\bar{\mathbf{y}})= & \bar{h}(y_1,\ldots,y_{d+1}) \\
  = & 10d+\sum_{i=1}^d[(g^{-1}(y_i))^2-10\cos(2\pi g^{-1}(y_i))]
  \end{align*}
  The global minimum of $\bar{h}$ occurs at $\bar{\mathbf{y}}_{1 \times (d+1)}^*=(\frac{1}{2d},\ldots,\frac{1}{2d}, \frac{1}{2})$ which follows from the fact that the global minimum of the original form of Rastrigin's function occurs at $\mathbf{x}^*=(x_1^*,\ldots,x_d^*)=(0,\ldots,0)$. Comparative study of performances of the algorithms are carried out for this function under the same setup as in (\ref{sect:ackley}). In Table \ref{table:big} it is noted that agian RMPSS outperforms other algorithms by a large margin. It is noted that for $d=5$, RMPSS (pl2) provides more accurate solution than GA with a 253 folds improvement in average computation time.
\begin{table}[]
\begin{adjustbox}{addcode={\begin{minipage}{1\width}}{\caption{Comparison of minimum value achieved and average computation time (in seconds) for solving transformed $d$-dimensional Ackley's function, Griewank's function and Rastrigin's function on simplex for $d = 5,10,25,50$ and $100$ using  RMPSS (deafult, lp1 \& lp2), SQP, IP and GA for $d=5,10,25,50,100$ starting from 100 randomly generated points in each case.}
\label{table:big}
\end{minipage}},rotate=90,center}
\resizebox{1.6\columnwidth}{!}{%
\bgroup
\def\arraystretch{1.5}
\begin{tabular}{|c|l|c|c|c|c|c|c|c|c|c|c|}
\hline
\multirow{2}{*}{Functions} & \multirow{2}{*}{Algorithms} & \multicolumn{2}{c|}{d = 5} & \multicolumn{2}{c|}{d = 10} & \multicolumn{2}{c|}{d = 25} & \multicolumn{2}{c|}{d = 50} & \multicolumn{2}{c|}{d = 100} \\ \cline{3-12} 
 &  & \begin{tabular}[c]{@{}c@{}}Min.\\ value\end{tabular} & \begin{tabular}[c]{@{}c@{}}Avg.\\ time\end{tabular} & \begin{tabular}[c]{@{}c@{}}Min.\\ value\end{tabular} & \begin{tabular}[c]{@{}c@{}}Avg.\\ time\end{tabular} & \begin{tabular}[c]{@{}c@{}}Min.\\ value\end{tabular} & \begin{tabular}[c]{@{}c@{}}Avg.\\ time\end{tabular} & \begin{tabular}[c]{@{}c@{}}Min.\\ value\end{tabular} & \begin{tabular}[c]{@{}c@{}}Avg.\\ time\end{tabular} & \begin{tabular}[c]{@{}c@{}}Min.\\ value\end{tabular} & \begin{tabular}[c]{@{}c@{}}Avg.\\ time\end{tabular} \\ \hline
\multirow{6}{*}{\begin{tabular}[c]{@{}c@{}}Ackley's Function\\ (transformed)\end{tabular}} & RMPSS & 3.24e - 02 & 0.099 & 1.16e - 01 & 0.199 & 3.74e - 01 & 0.546 & 9.24e - 01 & 1.189 & 1.34e - 00 & 3.685 \\ \cline{2-12} 
 & RMPSS (pl1) & 2.88e - 04 & 0.163 & 4.91e - 04 & 0.314 & 2.59e - 03 & 0.876 & 5.15e - 03 & 2.234 & 1.08e - 02 & 6.277 \\ \cline{2-12} 
 & RMPSS (pl2) & 1.31e - 06 & 0.186 & 6.45e - 06 & 0.394 & 2.47e - 05 & 1.194 & 4.97e - 05 & 3.063 & 1.04e - 04 & 9.307 \\ \cline{2-12} 
 & SQP & 1.65e - 00 & 0.032 & 2.01e - 00 & 0.060 & 4.71e - 00 & 0.192 & 1.73e - 00 & 0.541 & 1.27e - 00 & 2.861 \\ \cline{2-12} 
 & IP & 2.32e - 00 & 0.078 & 2.32e - 00 & 0.139 & 8.58e - 00 & 0.361 & 1.32e + 01 & 0.865 & 1.49e + 01 & 2.863 \\ \cline{2-12} 
 & GA & 8.48e - 04 & 38.395 & 3.60e - 00 & 40.653 & 1.12e + 01 & 40.097 & 1.44e + 01 & 39.761 & 1.60e + 01 & 46.626 \\ \hline
\multirow{6}{*}{\begin{tabular}[c]{@{}c@{}}Griewank's Function\\ (transformed)\end{tabular}} & RMPSS & 8.39e - 02 & 0.069 & 6.44e - 01 & 0.103 & 1.25e - 00 & 0.322 & 2.81e - 00 & 0.795 & 1.30e + 01 & 2.152 \\ \cline{2-12} 
 & RMPSS (pl1) & 7.65e - 03 & 0.111 & 8.87e - 03 & 0.213 & 6.24e - 03 & 0.599 & 2.60e - 02 & 1.518 & 9.12e - 02 & 5.609 \\ \cline{2-12} 
 & RMPSS (pl2) & 7.40e - 03 & 0.136 & 7.40e - 03 & 0.248 & 4.87e - 07 & 0.808 & 2.64e - 06 & 2.186 & 1.24e - 05 & 6.345 \\ \cline{2-12} 
 & SQP & 3.70e - 01 & 0.041 & 8.84e - 09 & 0.084 & 6.70e - 08 & 0.267 & 1.24e - 05 & 0.784 & 3.51e - 04 & 3.138 \\ \cline{2-12} 
 & IP & 2.04e - 01 & 0.102 & 2.33e - 09 & 0.174 & 2.47e - 08 & 0.294 & 1.86e - 07 & 0.586 & 1.23e - 06 & 1.802 \\ \cline{2-12} 
 & GA & 1.94e - 02 & 37.578 & 1.23e + 01 & 37.765 & 7.52e + 02 & 38.598 & 3.52e + 03 & 40.148 & 1.06e + 04 & 46.527 \\ \hline
\multirow{6}{*}{\begin{tabular}[c]{@{}c@{}}Rastrigin's Function\\ (transformed)\end{tabular}} & RMPSS & 1.25e - 00 & 0.068 & 2.68e - 00 & 0.125 & 9.21e - 00 & 0.300 & 1.60e + 01 & 0.859 & 4.77e + 01 & 2.349 \\ \cline{2-12} 
 & RMPSS (pl1) & 3.07e - 05 & 0.118 & 3.98e - 00 & 0.211 & 2.61e - 03 & 0.640 & 9.97e - 00 & 1.453 & 6.15e - 00 & 4.213 \\ \cline{2-12} 
 & RMPSS (pl2) & 1.51e - 09 & 0.163 & 3.98e - 00 & 0.333 & 2.55e - 07 & 0.929 & 9.95e - 01 & 2.084 & 5.97e - 00 & 6.014 \\ \cline{2-12} 
 & SQP & 2.98e - 00 & 0.030 & 2.19e + 01 & 0.062 & 1.79e + 02 & 0.220 & 4.05e + 02 & 0.718 & 6.85e + 02 & 3.578 \\ \cline{2-12} 
 & IP & 9.95e - 00 & 0.076 & 6.17e + 01 & 0.138 & 2.01e + 02 & 0.947 & 4.25e + 02 & 15.106 & 8.04e + 02 & 3.342 \\ \cline{2-12} 
 & GA & 6.8e - 06 & 41.271 & 3.24e + 01 & 40.772 & 4.68e + 02 & 40.730 & 1.88e + 03 & 42.553 & 5.32e + 03 & 45.797 \\ \hline
\end{tabular}
\egroup}
\end{adjustbox}
\end{table}

 \section{Application to analysis of Atlantic Hurricane Intensity Data using Simultaneous Quantile Regression}
 \label{sec_application}
 In 2008, \cite{Elsner2008} argued that the higher velocity hurricanes in the North Atlantic basin have got more stronger in the last couple of decades. In order to analyze the validity of this argument, we consider the Hurricane velocity data\footnote{\url{http://weather.unisys.com/hurricanes}} of North Atlantic region during 1981-2006 and we perform quantile regression, higher estimated quantile levels being our main point of interest. \cite{Das2017a} proposed a Bayesian method for quantile regression method using B-spline (\cite{Boor2001}) series for estimating the whole quantile curve simultaneously. The main challenge of estimating the whole quantile curve simultaneously remains in maintaining the monotonicity of the quantile curves, as any upper quantile curve should be always at the same level or higher than other lower quantile curves. After transforming the univariate predictor and the response variables to unit intervals by monotonic transformation, the estimation of the whole quantile curve in \cite{Das2017a} comes down to estimating two monotonically increasing differentiable functions $\xi_1(\cdot)$ and $\xi_2(\cdot)$ (see Appendix B for details) such that both map unit interval to another unit interval (i.e., each of them is a diffeomorphism of $[0,1]$ onto itself). \cite{Das2017a} used B-spline basis functions to estimate the functions $\xi_1(\cdot)$ and $\xi_2(\cdot)$ as monotonicity can easily be imposed in B-Splines by taking increasing coefficients of the B-spline basis functions (\cite{Boor2001}). 
 
 Suppose, we take equidistant knots $0=t_0 < t_1<\cdots < t_k = 1$ on the interval $[0,1]$ where $t_i = i/k$, $i=0,1,\ldots, k$ where $k$, a positive integer, denotes the number of divided components of unit interval. Let $h$ denote the degree of the B-spline (e.g., $h=2,3$ denotes the quadratic and cubic splines respectively). Let $\{B_{j,h}(t)\}_{j=1}^{k+h}$ denote the basis functions of $h$-th degree B-splines on unit interval. Then the B-spline expansion of the diffeomorphisms $\xi_1(\cdot)$ and $\xi_2(\cdot)$ are given by
 \begin{align}
     \xi_1(\tau) = \sum_{j=1}^{k+h}\theta_jB_{j,h}(\tau), \; 0=\theta_1 \leq \theta_2 \leq \cdots \leq \theta_{k+h} = 1, \nonumber \\
     \xi_2(\tau) = \sum_{j=1}^{k+h}\phi_jB_{j,h}(\tau), \; 0=\phi_1 \leq \phi_2 \leq \cdots \leq \phi_{k+h} = 1, \nonumber
 \end{align}
 where $\tau \in [0,1]$ denotes the quantile level. Now, define $$\gamma_j = \theta_{j+1} - \theta_{j}, \delta_j = \phi_{j+1}-\phi_j, \; j=1,\ldots,k+h-1.$$ Note that $$\gamma_j,\delta_j \geq 0, j=1,\ldots,k+h-1, \sum_{j=1}^{k+h-1}\gamma_j = \sum_{j=1}^{k+h-1}\delta_j = 1,$$ which implies $G =\{\gamma_j\}_{j=1}^{k+h-1}$ and $D =\{\delta_j\}_{j=1}^{k+h-1}$ are on unit simplexes. Therefore, given the dataset $\{(X_i,Y_i)\}_{i=1}^n$ ($n$ denotes sampl-size) the whole log-likelihood can be expressed as a function of two unit-simplexes $D$ and $G$, given by $$l\big(\{\theta_j\}_{j=1}^{k+h},\{\phi_j\}_{j=1}^{k+h}|(X_1,Y_1),\ldots,(X_n,Y_n)\big) = f(D,G).$$
 However, as discussed in \cite{Das2017a}, this log-likelihood does not have any closed form and evaluation of likelihood involves numerical integration, grid search and several other complex operations making it computationally expensive. For that reason it cannot be checked readily whether its concave or not analytically (see Appendix B for explicit form of the likelihood). Therefore, in order to avoid the computational burden of maximizing a (possibly) non-concave log-likelihood (which is equivalent to minimizing a non-convex function), \cite{Das2017a} considered Bayesian alternative to estimate $D$ and $G$ using posterior samples after incorporating Markov Chain Monte Carlo (MCMC) sampling. However, RMPSS can be used to maximize the log-likelihood since each set of unknown parameters $D$ and $G$ belongs to unit-simplex. 
 
 In order to analyze the North Atlantic Hurricane intensity data, we transform the $X$ (year) and $Y$ (velocities of Hurricane) to unit intervals in the same fashion as performed in \cite{Das2017a} and once the results are obtained, the estimated values are again transformed back to the original scale. Following the argument in \cite{Das2017a}, we take the degree of B-spline $h = 2$. The negative log-likelihood is minimized with RMPSS for the number of divided components of unit interval $k = 3,4, \ldots, 10$ and Akaike Information Criterion (AIC) is used to select the best possible value of $k$ which came up to be 5 for this dataset. Unlike the previous cases considered in this article, here two unit simplexes need to be estimated. Therefore, while using RMPSS to estimate $D$ and $G$, we update each simplex component alternatively within same iteration. In other words, while updating $D$ during $i$-th iteration, we fix the value of $G$ at $G^{(i-1)}$. Once $D$ is updated, we update $G$ by fixing the value of $D$ at $D^{(i)}$. In Figure \ref{fig:hurr_1} the estimated quantiles of the hurricane velocities are plotted during 1981--2006 along with the scatter-plot of the real data. Note that, for the higher estimated quantiles, the postive slope is more prominent than that of the lower quantiles which implies the higher velocity hurricanes have become more extreme over the time period while low velocity hurricanes tend to be similar during that period of investigation. Figure \ref{fig:hurr_2} shows how velocity of hurricane has changed over time at different quantile levels. Here, the gap between the velocities are higher at higher quantile and an increasing trend of velocities is observed over the time period. It should be noted that these plots looks similar to those estimated in \cite{Das2017a} using Bayesian methods.
 
  \begin{figure*}
     \subfigure[]{\includegraphics[width=0.5\textwidth]{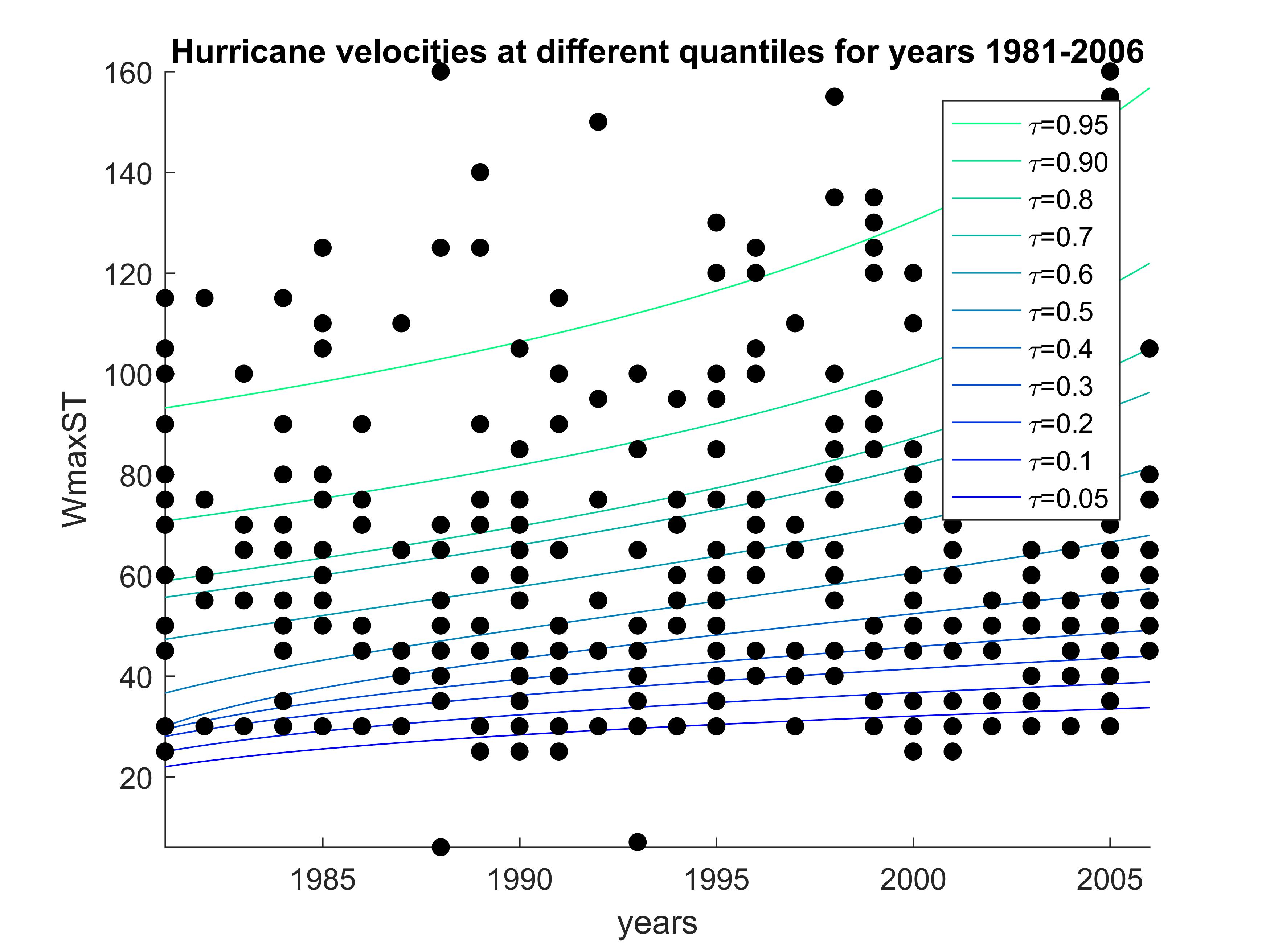}\label{fig:hurr_1}}
     \subfigure[]{\includegraphics[width=0.5\textwidth]{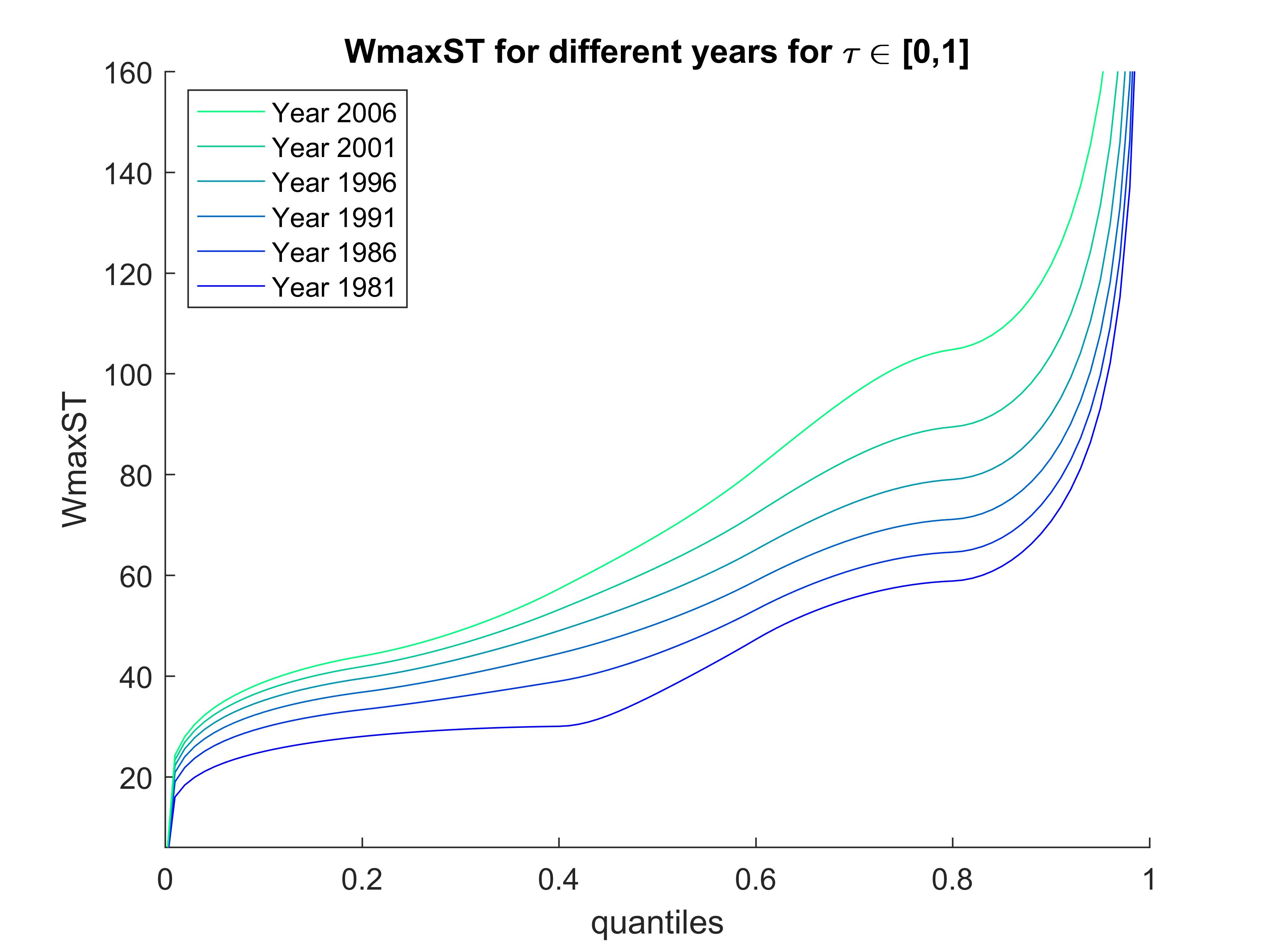}\label{fig:hurr_2}}
       \caption{(a) Estimated quantile curves (using RMPSS) of hurricane velocities during the period 1981--2006 for the quantiles $\tau = \{0.05,0.1,\ldots,0.9,0.95\}$, over the scatter-plot of hurricane velocities over that time period. (b) Hurricane velocities as a function of quantiles for years 1981, 1986, 1991, 1996, 2001, 2006.}
 \end{figure*}
 
 \underline{\it Parallel computation}~~ As mentioned in Section \ref{sec1} and \ref{sec_algo}, in RMPSS, within each iteration, since the objective function is evaluated in $2m$ (where $m$ denotes the dimension of the simplex parameter) independent directions, upto $2m$ parallel threads can be used. While using parallel computing, some time is spent to  distribute works to parallel threads and to collect (or combine) the results from those parallel threads after the computational tasks are over within all parallel threads. The time spent for allocating and collecting results while using parallel threads depends on the software\footnote{\url{https://www.mathworks.com/help/optim/ug/improving-performance-with-parallel-computing.html}}. In case the objective function is relatively inexpensive to compute, or in other words, if the computation time of evaluating the objective function is too small compared to the amount of extra time spent for parallel thread computing, little to no improvement in computation time might be gained for using parallel threading over single threading. However, in case the objective function is relatively expensive, the improvement of computation time using parallel threading becomes more perceivable since the time spent on allocation and collection of results from parallel threads does not differ much based on how expensive the objective function is; also as the objective function becomes more expensive to evaluate, that total amount of extra time required specifically for parallel threading starts getting smaller  compared to the total time spent for other computational operations. Thus, as the objective function becomes more expensive, more gain in computational time is obtained by using parallel over single threading.
 
 Since all the functions considered in the previous simulation scenarios are quite inexpensive in terms of computation time, we only consider single threading instead of parallel threading. But the log-likelihood considered in the North Atlantic Hurricane data analysis is computationally quite expensive. Since the total log-likelihood is the sum of the log-likelihood evaluated at each data-point, as the sample size increases, the computation time of the log-likelihood would also increase linearly with sample size. Thus, in case we may increase the sample size in the data, we can make the objective function (the negative of the log-likelihood) more time consuming. In the real data, the sample size is $380$. To provide a better understanding of the relation of the computation time as the computational time of the objective function increases, we consider one real and two artificial sample size scenarios denoted by \textit{Scenario 1}, \textit{Scenario 2} and \textit{Scenario 3}. In \textit{Scenario 1} we consider the sample size to be 380 as provided in the data-set. To make the evaluation of the log-likelihood function even more time consuming, in \textit{Scenario 2} we consider a dataset of size 7600 which is obtained by replicating the true dataset 20 times. In \textit{Scenario 3} we replicate the true sample 100 times to obtain the sample of size 38000. For the analysis, we consider quadratic B-spline ($k=2$) and for each scenario of sample sizes we consider 3 values of $h$ which are $\{9,19,39\}$; thus making the dimension of the each simplex parameter ($D$ and $G$) to be $m = k+h-1 = 10,20,40$. The main motivation of considering 3 different values of $m$ lies in the fact that as the size of parameter space increases, parallel computing is expected to perform faster compared to single threading because of distributing the job of function evaluations to multiple cores/threads instead of doing it using single core. For all the considered sample size scenarios, for those 3 above-mentioned values of $m$, we perform both single and parallel threading RMPSS. For parallel threading, we use desktop of same specification as mentioned in Section \ref{sec_app} with 12 threads incorporated by \texttt{parfor} loop in MATLAB 2016b. 
 \begin{table}[]
 \centering
 \resizebox{0.8\columnwidth}{!}{%
\begin{tabular}{|l|c|c|c|}
\hline
Experiments & $m$ & \begin{tabular}[c]{@{}c@{}}Time required\\ Parallel thread.\\ (seconds)\end{tabular} & \begin{tabular}[c]{@{}c@{}}Time required\\ Single thread.\\ (seconds)\end{tabular} \\ \hline
\multirow{3}{*}{Scenario 1} & 10 & 34.82 & 3.34 \\ \cline{2-4} 
 & 20 & 36.55 & 6.69 \\ \cline{2-4} 
 & 40 & 42.22 & 15.36 \\ \hline
\multirow{3}{*}{Scenario 2} & 10 & 56.79 & 63.52 \\ \cline{2-4} 
 & 20 & 66.63 & 128.95 \\ \cline{2-4} 
 & 40 & 92.46 & 298.79 \\ \hline
\multirow{3}{*}{Scenario 3} & 10 & 111.59 & 311.53 \\ \cline{2-4} 
 & 20 & 147.26 & 641.55 \\ \cline{2-4} 
 & 40 & 253.72 & 1472.10 \\ \hline
\end{tabular}}
\label{table:parallel}
\caption{Comparison of computation time to estimate the quantile regression curves for \textit{Scenario 1}, \textit{Scenario 2} and \textit{Scenario 3} with 1, 20 and 100 replications of North Atlantic Hurricane dataset (size $n =380$) respectively using parallel (12 threads) and single threading RMPSS. Computation times are noted in seconds.}
\end{table}
 The comparison of computation times are is provided in Table \ref{table:parallel}. Note that, the objective function being less expensive for \textit{Scenario 1}, single thread computing yields faster result. However, as the objective function becomes more challenging in \textit{Scenario 2} and \textit{Scenario 3}, an improvement in the performance of parallel threading is observed. We obtain upto 5.8 fold improvement in computation time in \textit{Scenario 3}. It is expected that more computational gain (upto 12 times, since 12 parallel threads are used) could be obtained for more expensive objective function scenarios. Also note that, within any \textit{Scenario}, the computation time for single threading increases almost at a linear rate as $m$ increases. However it is not true for parallel threading scenarios. Because as mentioned earlier, the total computational time for parallel threading is the sum of the buffer time (for distribution and collection of parameter values from parallel threads) and the computational time for other operations. Now, similar amount of buffer time is spent for all $m$ values but only the other portion of computational time increases as $m$ increases. As the complexity of the objective function increases through \textit{Scenario 2} and \textit{Scenario 3}, the proportion of this required buffer time becomes lesser compared to the total computational time; thus difference of computation time across different $m$ becomes more prominent as the objective function gets more expensive. That explains the fact that the difference of computation times using parallel threading across different $m$ values  is more prominent for \textit{Scenario 3} compared to \textit{Scenario 1} and \textit{Scenario 2}.

 \section{Discussion}
 \label{sec_discussion}
 In this paper, a novel efficient Blackbox optimization algorithm is proposed where the parameters belong to unit-simplex. RMPSS can be considered as a variation of pattern search where candidate solutions are generated in the neighborhood of the current solution by making movements across the coordinates. However, the main challenging aspect of designing RMPSS remains in maintaining the simplex constraint along with required pattern-search based update step while looking for a better solution starting from any given solution. Unlike existing pattern-search algorithms, the re-start strategy of \textit{run}s considered in RMPSS is shown to contribute in yielding better solutions while optimizing a wide range of objective functions. The proposed algorithm being derivative-free, unlike some derivative-based algorithms (e.g., SQP), it could be used efficiently to optimize functions even if the closed form of the derivative of the objective function does not exist or it is expensive to evaluate. Unlike some other global optimization techniques (e.g., GA), at each step, the number of candidate solutions generated are in the order of the dimension of the parameter space (i.e., $2m$ where $m$ is the dimension of the parameter) and the objective function value can be evaluated in parallel. Thus, in RMPSS, upto $2m$ parallel threads can be used making the computation even faster for optimizing expensive high dimensional objective functions. A study is also considered showing how parallel computing can be incorporated to yield solutions faster in case the objective function is relatively expensive and/or high-dimensional.
 
Another novelty in proposing RMPSS remains in the introduction of the sparsity parameter $\lambda$ which can be used to induce sparsity in case the solution is known to be sparse in prior. The sparse solution technique is useful for several statistical problems e.g., estimating mixture proportions of mixture model with low sample size and high number of possible classes/clusters. Under the regularity conditions (mentioned in Section \ref{sec_theory}), it is also shown that execution of a single \textit{run} is sufficient to reach the global minimum. So, in case the objective function is known to follow those regularity conditions, we can set $max\_runs=1$ and in that scenario, RMPSS can be used to minimize any convex function just like any other convex optimization algorithms (e.g., SQP) where no extra time is spent on looking for other possible solutions once a local minimum is reached. RMPSS is also shown to yield better solution in lesser time (in general) compared to several existing convex and Blackbox optimization techniques based on the challenging optimization problems of several low, medium and high-dimensional objective functions. Upto 250 fold improvement in computation time is observed in using RMPSS over GA. Several scenarios are also considered to give the readers an idea about how the accuracy of the solution can be improved by changing a few tuning parameters in case further computation time is affordable.
 
RMPSS is used to estimate the simultaneous quantiles of the North Atlantic Hurricane velocities during the period 1981--2006 using a simultaneous quantile regression method (\cite{Das2017a}) where the likelihood does not have any closed form expression and along with presence of (possibly) multiple modes. It is noted that the higher velocity hurricanes became more stronger over the period of study while the velocity of the hurricanes belonging to lower quantile did not show much of an increasing trend.

We include a brief discussion on how RMPSS can be used to estimate the proportion vector for parameter estimation problem in case of univariate finite mixture model in Appendix C. In future, RMPSS can be extended for the parameter space which consists of multiple unit-simplexes. This principle can also be extended for spherically constrained parameter spaces which would be useful for estimating fixed norm regression coefficients and in the context of directional statistics.
 \begin{acknowledgements}
 I would like to thank Dr. Rudrodip Majumdar, Dr. Debraj Das, Dr. Suman Chakraborty and Dr. Kushal Dey for helping me editing the earlier drafts of this paper and for their valuable suggestions for improvements. I would also like to acknowledge my Ph.D. adviser Dr. Subhashis Ghoshal for his valuable suggestions and suggested statistical problems which made me think of this algorithm. Also, I would like to thank the reviewers for their valuable suggestions.
 \end{acknowledgements}

 
 

 \newpage
\appendix  
\section*{Appendix A : Discussion on the number of operations and objective function evaluations required at each iteration of RMPSS}
\label{sec_appendixA}
Here we find the order of the number of basic operations and the number of objective function evaluations required at each iteration in terms of the dimension of the parameter space. Therefore, we find the upper bound of number of operations required for the worst case scenario which would be sufficient to determine the order of the number of basic operations. 

Suppose we want to minimize $f(\mathbf{p})$ where $\mathbf{p} \in \mathbf{S}$. At the beginning of each iteration, 4 arrays of length $m$, i.e.,  $\mathbf{s}^+, \mathbf{s}^-, \mathbf{f}^+, \mathbf{f}^+$ are initialized (see step (2) of STAGE 1 in Section \ref{sec_algo}). During each iteration, starting from the current value of the parameter $2m$ candidate solutions are generated in a such way that each of them belongs to the domain $\mathbf{S}$. Search algorithm for first $m$ of these movements have been described in step (3) of STAGE 1  of Section \ref{sec_algo}.     

 In step (3) of STAGE 1  of Section \ref{sec_algo}, note that it requires not more than $m$ operations to find $S_i^+$. To find $K_i^+$, it takes at most $m$ operations. As we are considering the worst case scenario in terms of maximizing the number or required operations, assume $K_i^+ \geq 1$. In step (3.1) and (3.2) of STAGE 1, suppose the value of $s_i^+$ is updated atmost $k$ times. So, we have $\frac{s_{initial}}{\rho^{k-1}} \leq \phi$ but $\frac{s_{initial}}{\rho^{k-2}} > \phi$. Hence $k = 1+\big[\frac{\log(\frac{s_{initial}}{\phi})}{\log(\rho)}\big]$, where $[x]$ returns the largest interger less than or equal to $x$. Corresponding to each update step of $s_i^+$, first it is checked whether $s_i^+ \leq \phi$ or not. It involves a single operation. Then deriving $\mathbf{q}_i^+$ involves not more than $2m$ steps because the most complicated scenario occurs for updating the positions of $\mathbf{q}_i^+$ which belong to $S_i^+$. And in that case, it takes total two operations for each site, one operation to find $\frac{s_i^+}{K_i^+}$ and one more to evaluate to subtract that quantity from $p_i^{(l)}$ for $l \in S_i^+$. In order to check whether $\mathbf{q}_i^+ \in \mathbf{S}$ or not, it requires $m$ operations. For the worst case scenario, we also add one more step required for updating $s_i^+ = \frac{s_i^+}{\rho}$. Hence the search procedure of any movement (i.e., for any $i \in \{1,\ldots,m\}$) in step (3) of STAGE 1 requires $m+m+k \times (1+2m+m+1) = m\times (2+3k)+2k$ operations. Hence, for $m$ movements (mentioned in step (3) of STAGE 1 in Section \ref{sec_algo}) it requires not more than $m^2\times (2+3k)+2mk$ operations. In a similar way, it can be shown that for step (4) also the maximum number of required operations is not more than $m^2\times (2+3k)+2mk$. 

In step (5) of STAGE 1 in Section \ref{sec_algo},  to find $k_1$ or $k_2$, it takes $(m-1)$ operations. The required number of steps for this step will be maximized if $\min (f_{k_1}^+,f_{k_2}^-) < Y^{(j)}$. Under this scenario, two more operations (i.e., comparisons) are required to find $\textbf{p}_{temp}$. So this step requires not more than $2 \times (m-1) + 2 = 2m$ operations. 

In step (6) of STAGE 1 in Section \ref{sec_algo}, it takes at most $m$ operations to find $S_{updated}$. In case $K_{updated}$ is not $m$, the required number of operations required in this step would be more than the case when $K_{updated}=m$. For finding out the number of operations required for the worst case scenario, assume $K_{updated} < m$. To find the value of $garbage$, maximum number of required steps is not more than $m$. Finally, it can be noted that in step (6.2), updating the value of the parameter of interest from $\mathbf{p}^{(j)}$ to $\mathbf{p}^{(j+1)}$ requires not more than $2m$ steps. So maximum number of operations required for step (6) of STAGE 1 is not more than $m+m+2m=4m$.

In step (7) of STAGE 1 in Section \ref{sec_algo}, to find $(\mathbf{p}^{(j)}(i) - \mathbf{p}^{(j-1)}(i))^2$ for each $i \in \{1,\ldots,m\}$, we need one operation for taking difference, and one operation for taking the square. Hence to find the sum of the squares, it needs $(3m-1)$ more operations. Comparing it's value with $tol\_fun$ takes one more operation. In the worst case scenario, it would take two more operations till the end of the iteration, i.e., update of $s^{(j)}$ at step (7) and it's comparison with $\phi$ at step (8). Hence after step (6), the required number of operations would be at most $(3m-1)+1+2 = 3m+2$. 

Hence for each iteration, in the worst case scenario, the number of required basic operations is not more than $m^2(2+3k)+2mk+2m+3m+(3m+2) = m^2(2+3k) + m(2k+8)+2$. So, number of basic operations required for each iteration in our algorithm  is of $O(m^2)$ where $m$ is the number of parameters to estimate. 

Note that the number of times the objective function is evaluated in each iteration is $2m+1$ (once at step (2), $m$ times at step (3.3) and $m$ times at step (4.3) of STAGE 1 in Section \ref{sec_algo}). Thus, we note that the order of the number of function evaluations at each iteration step is of $O(m)$.
\newpage

\section*{Appendix B : Model and Likelihood of Simultaneous Quantile Regression}
\label{sec_appendixB}
Let $Q_y(\tau|x)=\text{inf}\{q\, : \, P(Y\leq q| X=x)\geq \tau\}$ denote the $\tau$-th conditional quantile of a response $Y$ at $X=x$ for $0\leq\tau\leq 1$, where $X$ is the predictor. A linear simultaneous quantile regression model for $Q_y(\tau|x)$ at a given $\tau$ is given by $$Q_y(\tau|x)=\beta_0(\tau)+x\beta(\tau)$$ where $\beta_0(\tau)$ denotes the intercept and $\beta(\tau)$ denotes the slope which are smoothly varying function of $\tau$. After transforming the predictor and the response variables to unit variable by some monotonic transformation, as shown in \cite{Das2017a}, the linear quantile function can be represented as 
\begin{align}
Q_y(\tau|x)= x\xi_1(\tau)+(1-x)\xi_2(\tau) \; \text{for} \; \tau \in [0,1], \; x,y \in [0,1],
\label{our_eq}
\end{align}
for some functions $\xi_1(\tau)$ and $\xi_2(\tau)$ which are monotonically increasing in $\tau$ for $\tau \in [0,1]$ satisfying $\xi_1(0) = \xi_2(0) = 0$,  $\xi_1(1) = \xi_2(1) = 1$. Equation (\ref{our_eq}) can be re-framed as
\begin{align*}
Q_y(\tau|x)= \beta_0(\tau)+x\beta_1(\tau) \; \text{for} \; \tau \in [0,1], \; x,y \in [0,1],
\end{align*}
where $\beta_0(\tau) = \xi_2(\tau)$ and $\beta_1(\tau) = \xi_1(\tau)-\xi_2(\tau)$ denotes the slope and the intercept of the quantile regression. The conditional density for $Y$ is given by
 \begin{align}
 f_y(y|x) = \bigg(\frac{\partial}{\partial \tau}Q_y(\tau|x)|_{\tau=\tau_x(y)}\bigg)^{-1}=\bigg(\frac{\partial}{\partial \tau}\beta_0(\tau)+x\frac{\partial}{\partial \tau}\beta(\tau)|_{\tau=\tau_x(y)}\bigg)^{-1},
 \label{eq : density}
 \end{align}
 where $\tau_x(y)$ solves the equation 
 \begin{align}
x\xi_1(\tau)+(1-x)\xi_2(\tau) = y.
\label{tau_x_y}
 \end{align}
Therefore for any given dataset $\{(x_i,y_i)\}_{i=1}^n$, the likelihood is given by $\prod_{i=1}^n f_Y(y_i|x_i)$.

Let $0=t_0< t_1< \ldots < t_k=1$ be the equidistant knots on the interval $[0,1]$ such that $t_0 = 0$, $t_k=1$ and $t_i=\frac{1}{k}$ for all $i=0,1,\ldots,k-1$. Suppose $\{B_{j,h}(t)\}_{j=1}^{k+h}$ denote the basis functions of $h^{th}$ degree B-splines on [0,1] on the above mentioned set of equidistant knots. 
Now, the basis expansion of $\xi_1(\cdot)$ and $\xi_2(\cdot)$ are given by
\begin{align}
& \xi_1(\tau)=\sum\limits_{j=1}^{k+h} \theta_j B_{j,h}(\tau) \; \text{ where } \; 0=\theta_1<\theta_2<\cdots<\theta_{k+h}=1,  \nonumber \\
& \xi_2(\tau)=\sum\limits_{j=1}^{k+h} \phi_j B_{j,h}(\tau) \; \text{ where } \; 0=\phi_1<\phi_2<\cdots<\phi_{k+h}=1.
\label{eq:constraint}
\end{align}
Note that estimating $\{\theta_j,\phi_j\}_{j=1}^{k+h}$ is equivalent to estimating $G =\{\gamma_j\}_{j=1}^{k+h-1}, D = \{\delta_j\}_{j=1}^{k+h-1}$ where 
\begin{align*}
\gamma_j,\delta_j \geq 0,\; j=1,\cdots,k+h-1, \; \text{and} \; \sum_{j=1}^{k+h-1}\gamma_j= \sum_{j=1}^{k+h-1}\delta_j=1.
\end{align*}

\newpage
\section*{Appendix C : Application of RMPSS to estimate the membership probability vector of mixture model}
\label{sec_appendixC}
In this section we discuss how RMPSS can be used to estimate the proportion vector of a mixture model (e.g., Gaussian mixture). Suppose $X$ is a random variable which is coming from mixture of $C$ classes with density functions $\{f_j(\cdot|\theta_j)\}_{j=1}^C$ with probabilities $\mathbf{p} = (p_1,\ldots,p_C)$ such that $p_j\geq 0,\sum_{j=1}^Cp_j = 1$. So in this case, the density of the univariate mixture model is given by 
\begin{align}
   L( f(x|\mathbf{p},\boldsymbol{\theta}) = \sum_{j=1}^C p_jf_j(x|\theta_j),
\end{align}
where $\boldsymbol{\theta} = (\theta_1,\ldots, \theta_C)$ denotes the parameters of $C$ classes. For a given sample $\{x_i\}_{i=1}^n$, the likelihood is given by $$L(\mathbf{p},\boldsymbol{\theta}) = \prod_{i=1}^n\sum_{j=1}^C p_jf_j(x_i|\theta_j).$$ We consider the case where $\theta_j$s are univariate.\\
\underline{Case 1 : When $\boldsymbol{\theta}$ is known : }

In case, $\boldsymbol{\theta}$ is known, the likelihood is a function of only the proportion vector $\mathbf{p}$. Now, since $\theta_j$s are known, while estimating $\mathbf{p}$, problem of identifiability would not occur as the order of $\theta_j$s cannot change, so changing the order of elements of $\mathbf{p}$ would not produce the same likelihood value for all $x \sim X$ assuming all $\theta_j$s are different. So the proportion vector can be estimated using RMPSS without further modification.\\
\underline{Case 2 : When $\boldsymbol{\theta}$ is unknown : }

In case, $\boldsymbol{\theta}$ is also unknown, both $\mathbf{p} = (p_1,\ldots,p_C)$ and $\boldsymbol{\theta} = (\theta_1,\ldots,\theta_C)$ are needed to be estimated. However, unlike previous scenario, in this case, the likelihood function is not identifiable. For example, suppose $\mathbf{p}^{*}=(q_1,q_2,q_3), \mathbf{p}^{**}=(q_2,q_1,q_3), \boldsymbol{\theta}^{*} = (\phi_1,\phi_2, \phi_3), \boldsymbol{\theta}^{**} = (\phi_2,\phi_1, \phi_3)$ for $C=3$. Then note that $L(\mathbf{p}^{*},\boldsymbol{\theta}^{*}) = L(\mathbf{p}^{**},\boldsymbol{\theta}^{**})$ holds for any given sample. In order to get rid of the identifiability problem, we set a natural ordering of the $\boldsymbol{\theta}$ parameter space such that $\theta_1 > \theta_2 > \cdots > \theta_C$. Define, $\beta_j = \theta_j - \theta_{j-1}$ for $j=2,\cdots,C$. So, in case $\boldsymbol{\theta}\in \mathrm{R}^C$, then the new set of parameters are given by $\boldsymbol{\theta}^\prime = \{\theta_1,\beta_2,\beta_3,\ldots,\beta_C\} \in \mathrm{R}\times \mathrm{R^+}^{C-1}$. Now, to estimate the solution $\hat{\mathbf{p}}$ and $\hat{\boldsymbol{\theta}^\prime}$ for which the likelihood function is maximized, RMPSS can be used along with any other Black-box algorithm (e.g., GA, SA, PSO etc), such that at any given iteration, RMPSS is used to maximize the likelihood in terms of  $\mathbf{p}$ fixing the value of $\boldsymbol{\theta}^\prime$ at current value, then Genetic Algorithm (or any other Blackbox algorithm) is used to maximize the likelihood in terms of $\boldsymbol{\theta}^\prime$ fixing the value of $\mathbf{p}$ at the current updated value. Suppose $\mathbf{p}^{(k)}$ and $\boldsymbol{\theta^\prime}^{(k)}$ denote the updated values of $\mathbf{p}$ and $\boldsymbol{\theta^\prime}$ at the beginning of $k$-th iteration $(k \geq 2)$. Then following iteration steps should be performed iterartively until final solution is obtained.
\begin{itemize}
    \item While $L_1(\mathbf{p}^{(k)},\boldsymbol{\theta^\prime}^{(k)}) \neq L_1(\mathbf{p}^{(k-1)},\boldsymbol{\theta^\prime}^{(k-1)})$
    \begin{enumerate}
    \item Set $k = k+1$.
        \item Set $\mathbf{p}^{(k)} = \argmax_{\mathbf{p}}, L_1(\mathbf{p}, \boldsymbol{\theta^\prime} = \boldsymbol{\theta^\prime}^{(k-1)}$) by solving with RMPSS where $\mathbf{p} \in \Delta^{C-1}$,
        \item Set $\boldsymbol{\theta^\prime}^{(k)} = \argmax_{\boldsymbol{\theta^\prime}} L_1(\mathbf{p} = \mathbf{p}^{(k)}, \boldsymbol{\theta^\prime})$ by solving with GA (or any other Black-box optimization technique) where $\boldsymbol{\theta^\prime} \in \mathrm{R}\times \mathrm{R^+}^{C-1}$,
    \end{enumerate}
\end{itemize}
where $L_1(\mathbf{p}, \boldsymbol{\theta^\prime}) = L(\mathbf{p}, \boldsymbol{\theta})$ and
\begin{align*}
 \Delta^{C-1} = \{(p_{1}, \ldots, p_{C}) \in \mathbb{R}^{C} \; | \; p_i \geq 0, \;  i=1,\ldots,C, \; \sum_{i=1}^{C}p_i= 1\}.
 \end{align*}

\newpage

 \end{document}